\documentclass[a4paper,reqno,11pt]{amsart}
\usepackage{mathrsfs}
\usepackage{txfonts}
\usepackage{}
\usepackage{amsfonts}
\usepackage{amssymb}
\usepackage{amsfonts,amsmath,amsthm,amssymb,stmaryrd}

\newtheorem{Theorem}{Theorem}[section]
\newtheorem{Lemma}{Lemma}[section]
\newtheorem{Proposition}{Proposition}[section]

\newtheorem{Definition}{Definition}[section]

\numberwithin{equation}{section}
\usepackage[left=1 in, right=1 in,top=1 in, bottom=1 in]{geometry}

\def\XXint#1#2#3{{\setbox0=\hbox{$#1{#2#3}{\int}$ }
\vcenter{\hbox{$#2#3$ }}\kern-.6\wd0}}

\begin{document}
\title[ Global well-posedness of  MHD]{\bf Global well-posedness of  magnetohydrodynamic equations}
\author{Chengfei Ai}
\address{School of Mathematical Sciences, Xiamen University, Xiamen, 361005, China.}
\email[C.F. Ai]{aicf5206@163.com}
\author{Zhong Tan}
\address{School of Mathematical Sciences, Xiamen University, Xiamen, 361005, China.}
\email[Z. Tan]{ztan85@163.com}
\author{Jianfeng Zhou}
\address{School of Mathematical Sciences, Peking University, Beijing 100871, China}
\email[J. Zhou]{jianfengzhou@pku.edu.cn}
\thanks{Corresponding author:Jianfeng Zhou,\ jianfengzhou@pku.edu.cn}
\begin{abstract}
We study the global well-posedness of  magnetohydrodynamic (MHD) equations. The hydrodynamic system consists of the Navier-Stokes equations for the
fluid velocity coupled with a reduced from of the Maxwell equations for the magnetic field. The fluid velocity is assumed to satisfy a no-slip boundary
condition, while the magnetic field is subject to a time-dependent Dirichlet boundary condition. We first establish the global existence of weak and
strong solutions to (\ref{1.1})-(\ref{1.01}). Then we derive the existence of a uniform attractor for  (\ref{1.1})-(\ref{1.01}).
\bigbreak
\bigbreak
\end{abstract}
\maketitle
\section{Introduction}
We consider the following magnetohydrodynamic (MHD) equations in a bounded smooth domain $\Omega\subset\mathbb{R}^n$ ($n=2,3$):
\begin{equation}\label{1.1}
\begin{cases}
\partial_tu-\frac{1}{Re}\Delta u+u\cdot\nabla u-S(\nabla\times b)\times b+\nabla \tilde{p}=0&\text{in} \quad Q_T,\\
\partial_tb-\nabla\times(u\times b)+\frac{1}{Rm}\nabla\times(\nabla\times b)=0&\text{in} \quad Q_T,\\
{\rm div}~u=0, {\rm div}~ b=0 &\text{in} \quad Q_T,
\end{cases}
\end{equation}
subject to the initial-boundary conditions
\begin{align}
&u(x,0)=u_0(x),\ b(x,0)=b_0(x) \quad \text{in} ~\Omega,\label{1.2}\\
&u(x,t)=0,\ b(x,t)=h(x,t) \quad \text{on} ~ \Gamma_T,\label{1.02}
\end{align}
where $u_0$, $b_0$ satisfy the compatibility conditions:
\begin{equation}\label{1.01}
(u_0(x),b_0(x))|_{\Gamma}=(0,h)|_{t=0},\quad \textrm{div}~u_0=\textrm{div}~b_0=0.
\end{equation}
Here $\Gamma=\partial\Omega$, $0<T<\infty$, $Q_T:=\Omega\times [0,T]$, $\Gamma_T:=\Gamma\times[0,T]$. $Re>0$ is the Reynolds number, $Rm>0$
is the magnetic Reynolds number and $S=M^2/(ReRm)$ with $M$ be the Hartman number.  Furthermore, $u:Q_T\longrightarrow \mathbb{R}^3$,
$b:Q_T\longrightarrow \mathbb{R}^3$, $\tilde{p}=\tilde{p}(x,t)\in\mathbb{R}$ denote the
velocity of the fluid, the magnetic field and the pressure, respectively, while $u_{0}:\Omega\longrightarrow \mathbb{R}^3$,
$b_{0}:\Omega\longrightarrow \mathbb{R}^3$, $h:\Gamma_T\longrightarrow \mathbb{R}^3$ denote the given initial-boundary data with
${\rm div}~u=0={\rm div}~b$. When $n=2$, $\nabla \tilde{p}=(\partial_1,\partial _2,0)\tilde{p}$ and
$\nabla\times b=(\partial_2b_3,-\partial_1b_3,\partial_1b_2-\partial_2b_1)$.

The main purpose of this paper is  to investigate the well-posedness of (\ref{1.1})-(\ref{1.01}). We first review some previous works are related to MHD
equations. If $b=0$, then  (\ref{1.1}) reduces to the incompressible Navier-Stokes (NS) equations
\begin{equation*}
\begin{cases}
u_t-\frac{1}{Re}\Delta u+u\cdot\nabla u+\nabla \tilde{p}=0&\text{in} \quad Q_T,\\
\textrm{div}~u=0             &\text{in} \quad Q_T,\\
\end{cases}
\end{equation*}
There is a huge literature on the mathematical theory of the NS equations. Leray \cite{zhang24} first introduced the concept of weak solution and obtained
the existence of global weak solutions with $u_0\in L^2(\mathbb{R}^N)$ $(N\geq2)$ (see also \cite{zhang18}). Fujita et al. \cite{zhang17} derived the
well-posedness of the Cauchy problem with $u_0\in H^s(\mathbb{R}^N)$, ($N\geq\frac N2-1$) and $N\geq2$. Furthermore, there are many classical books,  for
example, Temam \cite{bo30}, Constantin--Foias \cite{re88} and  Lions \cite{re96}. References on the mild solutions and self-similar solutions in
$\mathbb{R}^3$ are the books by Cannone \cite{re95} and Meyer \cite{re99}. In particular, Jia and \v{S}ver\'{a}k \cite{jia} proved the classical Cauchy
problem for with $(-1)$-homogeneous initial data has a global scale-invariant solution which is smooth for positive times. For more details,
we  refer the reader to \cite{yin2,yin3,zhang6,zhang7,zhang15,zhang20,zhang25,zhang41} and the reference therein.

For the MHD system, the situation is more complicated because of the coupling effect between $u$ and $b$, and it has been the subject of many studies
by physicists and mathematicians due to its physicial importance, rich phenomena and mathematical challenges. The system (\ref{1.1}) was studied by Lions
et al. \cite{huang3}, the authors constructed a global weak solution and local strong solution to the initial boundary value problem. Furthermore, the
authors  also proved the existence of global strong solution for the small initial data. However, for the case of
large initial data, whether this unique local solution can exists globally is still a challenging open problem. Later, Temam and Sermange \cite{huang13}
(see also \cite{hhw, hxjde}) proved the regularity of weak solution $(u,b)\in L^{\infty}([0,T];H^1(\mathbb{R}^3))$. In addition, Kozono
\cite{zhang21} proved the existence of the classical solutions to (\ref{1.1}) in a bounded domain $\Omega\subset\mathbb{R}^3$. For suitable weak
solutions, He and Xin \cite{hx} (cf. \cite{twz}) obtained various partial regularity results. With mixed partial dissipation and additional magnetic
diffusion in $\mathbb{R}^2$, Wu et al. \cite{yin9} proved that the MHD system is globally well-posed for any data in $H^2(\mathbb{R}^2)$. For more
details, one can refer to \cite{zhang1,zhang2,yin22,huw,zhang31,zhang32,zhang33,swz,twz1,jiu,yin} the reference therein.

Without loss of generality, throughout the paper, we simply set $Re=Rm=S=1$, because the values of those coefficients do not play a role in the subsequent
analysis. Now, we define $p:=\tilde{p}+\frac12\nabla|b|^2$, and note that
\begin{align*}
&(\nabla \times b)\times b=b\cdot \nabla b-\frac{1}{2}\nabla(|b|^2),\\
&\nabla \times \nabla \times b=\nabla \textrm{div}~b-\Delta b,\\
&\nabla \times (u\times b)=b\cdot \nabla u-u\cdot \nabla b+u~\textrm{div}~b-b~\textrm{div}~u.
\end{align*}
Then the system (\ref{1.1}) can be rewritten as
\begin{equation}\label{1.3}
\begin{cases}
\partial_tu-\Delta u+u\cdot\nabla u-b\cdot\nabla  b+\nabla p=0&\text{in} \quad Q_T,\\
\partial_tb-\Delta b+u\cdot\nabla b-b\cdot\nabla u=0&\text{in} \quad Q_T,\\
\textrm {div}~u=0, \textrm {div}~ b=0 &\text{in} \quad Q_T,
\end{cases}
\end{equation}
When considering the technically more challenging case of time-dependent Dirichlet boundary data $h:\Gamma_T\longrightarrow\mathbb{R}^3$ for $b$, this
turns out to be a challenging task, since the boundary data $h$ will lead to several new difficulties, e.g., one can not obtain the energy estimates
directly.  In order to avoid this flaw, some lifting functions will be introduced (see Section \ref{se2}). The main purpose of this paper is divided into
several points:
\begin{enumerate}
  \item  We prove the global existence of weak solutions to (\ref{1.2})-(\ref{1.3}) for $n=2,3$, and strong solutions for $n=2$,
  instead of using the contraction mapping principle in \cite{w4}, here, we  employ the semi-Galerkin approximation method (see Section \ref{se3})
  to establish the existence of weak and strong solutions.
  \item  If $n=2$, we prove the continuous dependence of boundary-initial data and the uniqueness of weak-strong solutions;
  \item  If $n=2$, we obtain the existence of a uniform attractor for (\ref{1.2})-(\ref{1.3}).
\end{enumerate}

\noindent{\textbf{Notation}}. Throughout this paper, $c$ denotes a general constant may vary in different estimate. If the dependence need to be
explicitly stressed, some notations like $c_0,$ $c_1,$ $c(n)$ will be used. As usual, $L^p(\Omega)$, $W^{k,p}(\Omega)$ stand for the Lebesgue and Sobolev
spaces with $k\geq0$ and $p\geq1$. In particular, we denote  $W^{k,2}(\Omega)$ by $H^k(\Omega)$. Meanwhile,  we will use the shorthand notions
$\|\cdot\|_{L^2}$, $\|\cdot\|_{H^1}, \cdots$ instead of the norms defined in the domain $\Omega$, namely, $\|\cdot\|_{L^2(\Omega)}$,
$\|\cdot\|_{H^1(\Omega)}, \cdots$. Moreover, we set
\begin{align*}
\mathcal {D}=& \left\{v:~v\in C_0^{\infty}(\Omega,\mathbb{R}^3),\ {\rm div}~v=0\right\}, \\
H=& \text{closure of}~ \mathcal {D}~ \text{in}~ L^2(\Omega,\mathbb{R}^3),\\
V=& \text{closure of}~ \mathcal {D}~ \text{in}~ H_0^1(\Omega,\mathbb{R}^3), \\
V'=&\text{the dual of}~ V.
\end{align*}

Our main results are stated in the following theorems.
\begin{Theorem}\label{th1.1}
Let $T>0$,  $\Omega\subset\mathbb{R}^n$ ($n=2,3$) be a smooth bounded domain.  Suppose that
\begin{equation}\label{1.6}
h\in L^{q_n}([0,T];H^{\frac12}(\Gamma)),\quad \partial_th\in L^2([0,T];H^{-\frac12}(\Gamma)),
\end{equation}
where $q_n=4$ for $n=2$ and $q_n=8$ for $n=3$, $(u_0,b_0)\in H\times L^2(\Omega)$ with $u_0,b_0$
satisfy the compatibility condition (\ref{1.01}). Then  the problem (\ref{1.2})-(\ref{1.3}) admits a global weak solution $(u,b)$ such that
\begin{equation*}
(u,b)\in L^{\infty}([0,T];H\times L^2(\Omega))\cap L^{2}([0,T];V\times H^1(\Omega)).
\end{equation*}
In particular, if $n=2$, the problem (\ref{1.2})-(\ref{1.3}) admits a unique global weak solutions.
\end{Theorem}
Due to the time-dependent boundary condition (\ref{1.02}), the system (\ref{1.2})-(\ref{1.3}) no longer satisfies the dissipative energy law like the
autonomous case (see e.g. \cite{zhang32,zhang33}). However, by the lifting function $h_E$ (see (\ref{2.1})), we can also obtain a specific energy
inequality (\ref{3.20}). This, together with Lemma \ref{le2.2} implies   a uniform estimates for global weak solutions to (\ref{1.2})-(\ref{1.3})
on $Q_T$.

Based on Theorem \ref{th1.1}, under more regular assumptions for initial-boundary data we can further prove the existence of a unique global strong
solution to (\ref{1.2})-(\ref{1.3}) in two spatial dimensions.
\begin{Theorem}\label{th1.2}
Let $\Omega\subset\mathbb{R}^2$  be a smooth bounded domain.  Suppose that
\begin{equation}\label{1.8}
h\in L^2([0,T];H^{\frac32}(\Gamma)),\quad \partial_th\in L^2([0,T];H^{-\frac12}(\Gamma)),
\end{equation}
$(u_0,b_0)\in V\times H^1(\Omega)$ with $u_0,b_0$
satisfy the compatibility condition (\ref{1.01}). Then for any $T>0$, the problem (\ref{1.2})-(\ref{1.3})
admits a unique global strong solution $(u,b)$ such that
\begin{equation}\label{1.9}
(u,b)\in L^{\infty}([0,T];V\times H^1(\Omega))\cap L^{2}([0,T];H^2(\Omega)\times H^2(\Omega)).
\end{equation}
\end{Theorem}
As a consequence, from (\ref{1.8})-(\ref{1.9}), one can easily verify that
\begin{equation}\label{1.08}
(\partial_tu,\partial_td)\in L^2([0,T];H\times L^2(\Omega)).
\end{equation}
With the help of the interpolation (cf. \cite{wu43}), then (\ref{1.08}) implies the continuity of $(u,b)$, i.e.,  $(u,b)\in C([0,T];V\times H^1(\Omega))$.

Finally, according to Definitions  \ref{de4.1}-\ref{de4.6} and the existence of weak and strong solutions, we can derive the existence of a uniform
attractor for (\ref{1.3})
\begin{Theorem}\label{th1.3}
Let $\Omega\subset\mathbb{R}^2$  be a smooth bounded domain. Let all assumptions of Theorem \ref{th1.2} and (A2) (see Section \ref{se4.1}) be verified.
Then the process $\left\{U_h(t,\tau)\right\}$ generated by the solution operator of (\ref{1.3}) admits a compact
uniform (w.r.t. $h\in \Sigma_1$) attractor $\mathcal {A}_{\Sigma_1}$ in $V\times H^1$, which uniformly (w.r.t. $h\in \Sigma_1$)
attracts the bounded sets in $X$. Furthermore, there holds
\begin{equation*}
\mathcal {A}_{\Sigma_1}=\bigcup_{h\in \Sigma_1}\mathcal {K}_h(0),
\end{equation*}
where $\mathcal {K}_h$ is the kernel of the process $\left\{U_h(t,\tau)\right\}$ and $\mathcal {K}_h$ is nonempty for all $h\in \Sigma_1$.
\end{Theorem}
The rest of our paper is organized as follows. First of all, in Section \ref{se2}, we present some useful lemma which will be heavily used in our proof.
Next, in Section \ref{se3}, we prove the global existence of weak solutions and strong solutions to (\ref{1.2})-(\ref{1.3}). Further, for $n=2$, we also
derive the continuous dependence of initial-boundary data and uniqueness of weak-strong solutions. Finally, we obtain the existence of a uniform attractor
for (\ref{1.2})-(\ref{1.3})
\section{Preliminary}\label{se2}
Throughout this section, we collect some helpful results, some of which have been proven elsewhere. The following is a  regularity result for the Stokes
problem (see e.g., \cite{wu47} Chapter 1, Proposition 2.2).
\begin{Lemma}\label{le2.1}
Suppose the Stokes operator $S:~D(S)=V\cap H^2(\Omega)\longrightarrow H$ defined by
\begin{equation*}
Su=-\Delta u+\nabla P\in H,\quad \forall u\in D(S),
\end{equation*}
where $P\in H^1(\Omega)$. Then it holds that
\begin{equation*}
\|u\|_{H^2}+\|P\|_{H^1/\mathbb{R}}\leq c\|Su\|_{L^2},\quad \forall u\in D(S),
\end{equation*}
where $c=c(n,\Omega)$.
\end{Lemma}
In order to deal with the non-autonomous boundary term and obtain proper energy estimates for global  solutions, we introduce some suitable  lifting
functions. The first lifting problem for (\ref{1.3}) is defined by:
\begin{equation}\label{2.1}
\begin{cases}
-\Delta h_E=0&\text{in} ~\Omega,\\
h_E=h(x,t) &\text{on}~ \Gamma.
\end{cases}
\end{equation}
Taking into account the classical elliptic regularity theory (see e.g., \cite{wu39,wu46}), we have the  existence and regularity result:
\begin{Lemma}\label{le2.2}
Suppose that $h$ satisfies (\ref{1.6}), then the lifting problem (\ref{2.1}) admits  a unique solution $h_E$ such that
\begin{equation*}
h_E\in H^1([0,T];L^2(\Omega))\cap L^{\infty}([0,T];H^{\frac12}(\Omega))\cap L^2([0,T];H^{1}(\Omega)).
\end{equation*}
Furthermore, for  $t\in [0,T]$  the following regularity results hold:
\begin{align*}
\int_{0}^{t}\|h_E(\tau)\|_{H^{k+1}}^{2}dt&\leq c\int_{0}^{t}\|h(\tau)\|_{H^{k+\frac{1}{2}}(\Gamma)}^{2}d\tau;\\
\int_{0}^{t}\|\partial_t h_E(\tau)\|_{H^k}^{2}dt&\leq c\int_{0}^{t}\|\partial_t h(\tau)\|_{H^{k-\frac{1}{2}}(\Gamma)}^{2}d\tau,
\end{align*}
where $k=0,1$.
\end{Lemma}
The second lifting problem for (\ref{1.3}) has the following parabolic type:
\begin{equation}\label{2.2}
\begin{cases}
\partial_th_p-\Delta h_p=0&\text{in} ~Q_T,\\
h_p(x,0)=b_0(x) &\text{in} ~ \Omega,\\
h_p(x,t)=h(x,t) &\text{on} ~ \Gamma_T.
\end{cases}
\end{equation}
From the standard theory of linear parabolic system (see e.g., \cite{wu39}), the following results hold:
\begin{Lemma}\label{le2.3}
(1) Suppose that  $b_0\in L^2(\Omega)$ and (\ref{1.01}), (\ref{1.6}) hold. Then for $t\in [0,T]$, there exists a unique weak solution to (\ref{2.2})
\begin{equation*}
 h_p\in L^{\infty}([0,T];L^{2}(\Omega))\cap L^2([0,T];H^{1}(\Omega)),
\end{equation*}
and  the following estimate holds
\begin{equation}\label{2.3}
\|h_p(t)\|_{L^2}^{2}+\int_{0}^{t}\|\nabla h_p\|_{L^2}^{2}dt\leq\|b_0\|_{L^2}^{2}+c\int_{0}^{t}\|h\|_{H^\frac{1}{2}(\Gamma)}^{2}dt.
\end{equation}
(2) Let $b_0\in H^1(\Omega)$ and (\ref{1.01}), (\ref{1.8}) are satisfied.  Then  (\ref{2.2}) admits a unique strong solution
\begin{equation*}
h_p\in L^{\infty}([0,T];H^{1}(\Omega))\cap L^2([0,T];H^{2}(\Omega)).
\end{equation*}
Furthermore, for all $t\in [0,T]$, there holds
\begin{equation}\label{2.4}
\|h_p\|_{H^{1}}^{2}+\int_{0}^{t}\|h_p(\tau)\|_{H^{2}}^{2}d\tau\leq\|b_0\|_{H^{1}}^{2}+c\int_{0}^{t}(\|\partial_t h(\tau)\|_{H^{-\frac{1}{2}}(\Gamma)}^{2}
+\|h(\tau)\|_{H^\frac{3}{2}(\Gamma)}^{2})d\tau.
\end{equation}
\end{Lemma}
In addition, we shall use a interpolation and we formulate it in the form we need (cf. \cite{bo1}).
\begin{Lemma}\label{le2.4}
Let $\Omega\subset\mathbb{R}^2$ with compact smooth boundary and $g\in H^2(\Omega)$,  then
\begin{equation*}
\|g\|_{L^{\infty}}\leq c\|g\|_{H^1}\left(1+\ln\frac{\|g\|_{H^2}^{2}}{\|g\|_{H^1}^{2}}\right)^{\frac{1}{2}},
\end{equation*}
where the constant $c$ depends only on the domain $\Omega$.
\end{Lemma}
\section{Well-posedness of (\ref{1.3})}\label{se3}
In this section, we are devoted to proving the global existence of weak solution $(n=2,3)$ and  strong solution $(n=2)$ to (\ref{1.2})-(\ref{1.3}).
\subsection{Existence of weak solution}\label{se3.1}
In order to  prove the global existence of weak solution to (\ref{1.2})-(\ref{1.3}), we will use a  semi-Galerkin approximation scheme similar to
\cite{bo18} with some necessary modifications. Precisely,  we will use the usual Faedo-Galerkin method only for the velocity field $u$. Let the family
$\{\xi_i\}_{i=1}^{\infty}$  be a  basis of $V$, which is given by eigenfunction of  the Stokes problem:
\begin{equation}\label{3.01}
\Delta\xi_i+\nabla p_i=-\lambda_i\xi_i,\quad \xi_i|_{\Gamma}=0,
\end{equation}
where $\lambda_i$ is the  eigenvalue corresponding to $\xi_i$. Here,  $0<\lambda_1\leq\lambda_2\leq\cdots\leq\lambda_i\leq\cdots$, with
$\lambda_i\rightarrow\infty$ as $i\rightarrow\infty.$  For every $m\in \mathbb{N}$, we denote by $V^m=span\{\xi_1,\cdots,\xi_m\}$ be the finite
dimensional subspace of $V$. At this stage, for any $m\in \mathbb{N}$ and $T>0$, we consider the following approximate problem:
\begin{equation}\label{3.1}
\begin{cases}
\partial_t u_m-\Delta u_m+u_m\cdot\nabla u_m-b_m\cdot\nabla b_m+\nabla p_m=0 &\text{in}~    Q_T,\\
\partial_t b_m-\Delta b_m+u_m\cdot\nabla b_m-b_m\cdot\nabla u_m=0 &\text{in}~ Q_T,\\
\textrm{div}~u_m=\textrm{div}~b_m=0 &\text{in}~ Q_T,\\
u_m(0)=u_{0m}:= P_m u_0,~b_m(x,0)=b_0 &\text{in}~ \Omega,\\
u_m(x,t)=0,~b_m(x,t)=h(x,t)&\text{on}~ \Gamma_T.
\end{cases}
\end{equation}
Here, $P_m$ denotes the orthogonal projection from $H$ onto $V^m$.

Before proving the global existence of weak solutions to (\ref{3.1}), we propose to prove the local time  existence of $(u_m,b_m)$.
\begin{Proposition}\label{po3.1}
Let all assumptions in Theorem \ref{th1.1} be verified. For every $m\in \mathbb{N}$, there is a time $T_1\in (0,T]$ depending on $u_0$, $b_0$,
$m$ and $\Omega$ such that (\ref{3.1}) admits a unique  solution $(u_m,b_m)$ on $[0,T_1]$ satisfying
\begin{equation}\label{3.2}
u_m\in H^1([0,T_1];V^m),\quad b_m\in L^{\infty}([0,T_1]; L^2(\Omega))\cap L^2([0,T_1];H^1(\Omega)).
\end{equation}
\end{Proposition}
\begin{proof}
We start by choosing an arbitrary vector $\bar{u}_m\in C([0,T]; V^m)$ such that
\begin{equation}\label{3.3}
\sup_{t\in [0,T]}\|\bar{u}_m\|_{L^2}^2\leq M,
\end{equation}
where $M>0$ be a constant satisfying $\|u_0\|_{L^2}^2\leq \frac M2$.

We will now apply a fixed point argument to prove the existence of weak solution $(u_m,b_m)$ to (\ref{3.1}). First, we consider the following splitting
\begin{description}
\item[(i)] Let $\bar{u}_m\in C([0,T]; V^m)$ be a given velocity field, we look for $b_m \in L^{\infty}([0,T];L^2(\Omega))\cap L^2([0,T];H^1(\Omega))$
be the solution to the following problem:
\begin{equation}\label{3.4}
\begin{cases}
\partial_t b_m-\Delta b_m+\bar{u}_m\cdot\nabla b_m-b_m\cdot\nabla \bar{u}_m=0 &\text{in}~ Q_T,\\
\textrm{div}~b_m=0&\text{in}~ Q_T,\\
b_m(x,0)=b_0 &\text{in}~ \Omega,\\
b_m(x,t)=h(x,t) &\text{on}~ \Gamma_T.
\end{cases}
\end{equation}
\item[(ii)] Let $b_m\in L^{\infty}([0,T];L^2(\Omega))\cap L^2([0,T];H^1(\Omega))$ be the magnetic field just determined by (\ref{3.4}), we turn to look
for $u_m=\sum_{i=1}^mg_i^m(t)\xi_i(x)$  which solves the following problem:
\begin{equation}\label{3.5}
\begin{cases}
\partial_t u_m-\Delta u_m+\bar{u}_m\cdot\nabla u_m-b_m\cdot\nabla b_m+\nabla p_m=0 &\text{in}~  Q_T,\\
u_m(x,0)=u_{0m} &\text{in}~   \Omega,\\
u_m(x,t)=0   &\text{in}~\Gamma_T.
\end{cases}
\end{equation}
\end{description}
In what follows, for simplicity, we divide the proof into several steps.

\textbf{Step 1.} Existence and uniqueness for (\ref{3.1}). Define $\hat{b}_m:=b_m-h_p$ and $h_p$ is the solution of the lifting problem (\ref{2.2}), then
(\ref{3.4}) can be rewritten as
\begin{equation}\label{3.6}
\begin{cases}
\partial_t \hat{b}_m-\Delta \hat{b}_m+\bar{u}_m\cdot\nabla b_m-b_m\cdot\nabla \bar{u}_m=0 &\text{in}~ Q_T,\\
\hat{b}_m(x,0)=0 &\text{in}~ \Omega,\\
\hat{b}_m(x,t)=0 &\text{on}~ \Gamma_T.
\end{cases}
\end{equation}
By energy estimate, for (\ref{3.6})$_1$ on $\hat{b}_m$, we have
\begin{align}\label{3.7}
\frac12\frac{d}{dt}\|\hat{b}_m\|_{L^2}^2+ \|\nabla\hat{b}_m\|_{L^2}^2=&\int_{\Omega} [(\hat{b}_m+h_p)\cdot\nabla\bar{u}_m]\cdot\hat{b}_mdx
-\int_{\Omega} [\bar{u}_m\cdot\nabla(\hat{b}_m+h_p)]\cdot\hat{b}_mdx\nonumber  \\
\leq&\|\hat{b}_m\|_{L^4}^2\|\nabla\bar{u}_m\|_{L^2}+\|\nabla\bar{u}_m\|_{L^2}\|\hat{b}_m\|_{L^4}\|h_p\|_{L^4}
 +\|\bar{u}_m\|_{L^4}\|\hat{b}_m\|_{L^4}\|\nabla h_p\|_{L^2}\nonumber  \\
\leq&c\|\hat{b}_m\|_{L^2}^{2\theta}\|\nabla\hat{b}_m\|_{L^2}^{2(1-\theta)}\|\nabla\bar{u}_m\|_{L^2}
+c\|\nabla\bar{u}_m\|_{L^2}\|\nabla\hat{b}_m\|_{L^2}\|h_p\|_{H^1}\nonumber  \\
\leq&cM^q\|\hat{b}_m\|_{L^2}^2+cM^2\|h\|^2_{H^{\frac12}(\Gamma)}+\frac12\|\nabla\hat{b}_m\|_{L^2}^2,
\end{align}
where $\theta=\frac 12$, $q=2$ when $n=2$ and $\theta=\frac 14$, $q=4$ when $n=3$.

Applying the Gronwall's inequality, then (\ref{3.7}) implies that for any $t\in [0,T]$
\begin{equation}\label{3.8}
\|\hat{b}_m(t)\|_{L^2}^2+\int_0^t\|\nabla\hat{b}_m(\tau)\|_{L^2}^2d\tau\leq cM^2(cM^qt+1)e^{cM^qt}\|h\|^2_{L^2_tH_x^{\frac12}(\Gamma_t)}.
\end{equation}
With the help of the a priori estimate above, now, we proceed to prove the local existence of solution $b_m$ to (\ref{3.4}). First, we construct the
solution sequence $(b_m^j)_{j\geq0}$ by solving iteratively the following scheme for $j\geq0$:
\begin{equation}\label{3.9}
\begin{cases}
\partial_t b_{m}^{j+1}-\Delta b_{m}^{j+1}+\bar{u}_m\cdot\nabla b_{m}^{j+1}-b_{m}^{j}\cdot\nabla\bar{u}_m=0 &\text{in}~Q_T,\\
\textrm{div}~ b_{m}^{j+1}= \textrm{div}~ b_{m}^{j}=0&\text{in}~Q_T,\\
b_{m}^{j+1}(x,0)=b_0 &\text{in}~\Omega,\\
b_{m}^{j+1}(x,t)=h &\text{on}~\Gamma_T,
\end{cases}
\end{equation}
where $b_m^0=0$ is set at initial step. Without loss of generality, taking $T_0$ suitable small, by induction we shall prove that there are constants
$\tilde{M}>0$ depending on $b_0$, $h$ and $m$ such that
\begin{equation}\label{3.10}
\|b_{m}^j(t)\|_{L^2}^2+\int_0^t\|\nabla b_{m}^j(\tau)\|_{L^2}^2d\tau\leq \tilde{M},
\end{equation}
for all $t\in [0,T_0]$. In fact, suppose that (\ref{3.10}) is true for some $j\geq0$, then similar to (\ref{3.7})-(\ref{3.8}), we can see that
\begin{align}\label{3.11}
\frac12\frac{d}{dt}\|\hat{b}_{m}^{j+1}\|_{L^2}^2+\|\nabla \hat{b}_{m}^{j+1}\|_{L^2}^2
=&-\int_{\Omega}(b_{m}^j\otimes \bar{u}_m)\cdot\nabla \hat{b}_{m}^{j+1}dx
-\int_{\Omega}(\bar{u}_m\cdot\nabla b_{m}^{j+1})\cdot \hat{b}_{m}^{j+1}dx\nonumber\\
\leq& c\|b_{m}^j\|_{L^2}^2\|\bar{u}_m\|_{L^{\infty}}^2
+c\|\nabla\bar{u}_m\|_{L^2}^2\|h_p\|_{H^1}^2+\frac12\|\nabla \hat{b}_{m}^{j+1}\|_{L^2}^2\nonumber\\
\leq& c\|b_{m}^j\|_{L^2}^2\|\bar{u}_m\|_{L^2}^2
+c\|\nabla\bar{u}_m\|_{L^2}^2\|h\|_{H^{\frac12}(\Gamma)}^2+\frac12\|\nabla \hat{b}_{m}^{j+1}\|_{L^2}^2
\end{align}
where $\hat{b}_m^j=b_m^j-h_p$,  and in the last inequality, we have taken into account that
$\|\bar{u}_m\|_{L^{\infty}}\leq c(m)\|\bar{u}_m\|_{L^2}$, since $V^m$ is finite dimensional. Hence, from (\ref{3.11}),  it follows that
\begin{equation*}
\|\hat{b}_{m}^{j+1}(t)\|_{L^2}^2+c\int_0^t\|\nabla\hat{b}_{m}^{j+1}(\tau)\|_{L^2}^2d\tau
\leq c\tilde{M}\|\bar{u}_m\|_{L_t^{\infty}L_x^2(Q_T)}t
+c\|\nabla\bar{u}_m\|_{L_t^{\infty}L_x^2(Q_T)}\int_0^t\|h(\tau)\|_{H^{\frac12}(\Gamma)}^2d\tau,
\end{equation*}
for all $t\in [0,T_0]$. This together with (\ref{1.6}) implies that, for proper small constant $T_0>0$ and large constant $\tilde{M}>0$, (\ref{3.10})
is true for $j+1$ and hence it holds for all $j\geq0$.

Next, we shall show the convergence of the sequence $(b^j)_{j\geq0}$. By taking the difference of (\ref{3.9}) for $j$ and $j+1$, we have for $j\geq1$
\begin{equation}\label{3.12}
\begin{cases}
\partial_t\bar{b}_{m}^{j+1}-\Delta \bar{b}_{m}^{j+1}+\bar{u}_m\cdot\nabla\bar{b}_{m}^{j+1}-\bar{b}_{m}^{j}\cdot\nabla\bar{u}_m=0 &\text{in}~Q_T,\\
{\rm div}~\bar{b}_{m}^{j}={\rm div}~\bar{b}_{m}^{j+1}=0&\text{in}~Q_T,\\
\bar{b}_{m}^{j+1}(x,0)=0 &\text{in}~\Omega,\\
\bar{b}_{m}^{j+1}(x,t)=0 &\text{on}~ \Gamma_T,
\end{cases}
\end{equation}
where $\bar{b}_{m}^{j}=\bar{b}_{m}^{j}-\bar{b}_{m}^{j-1}$.

Employing the energy estimate, for (\ref{3.12}) on $\bar{b}_{m}^{j+1}$, we arrive at
\begin{align*}
\frac12\frac{d}{dt}\|\bar{b}_{m}^{j+1}\|^2_{L^2}+\|\nabla\bar{b}_{m}^{j+1}\|^2_{L^2}
&\leq\frac12\|\nabla{b}_{m}^{j+1}\|^2_{L^2}+\frac12\int_{\Omega}|\bar{b}_{m}^j|^2|\bar{u}_m|^2dx \\
&\leq\frac12\|\nabla{b}_{m}^{j+1}\|^2_{L^2}+c\|\bar{u}_m\|^2_{L^{\infty}}\|\bar{b}_{m}^j\|^2_{L^2} \\
&\leq\frac12\|\nabla{b}_{m}^{j+1}\|^2_{L^2}+c(m)\|\bar{u}_m\|^2_{L^{2}}\|\bar{b}_{m}^j\|^2_{L^2}.
\end{align*}
The further time integration gives
\begin{equation}\label{3.13}
\|\bar{b}_{m}^{j+1}(t)\|^2_{L^2}+\int_0^t\|\nabla\bar{b}_{m}^{j+1}(\tau)\|^2_{L^2}d\tau\leq c\|\bar{u}_m\|^2_{L_t^{\infty}L^{2}_x(Q_T)}
\sup_{t\in[0,T_0]}\|\bar{b}_{m}^{j}(t)\|^2_{L^2}T_0,
\end{equation}
for all $t\in [0,T_0]$. By the smallness of $T_0$, we can see that there exists a constant $\lambda\in (0,1)$ such that
\begin{equation}\label{3.14}
\sup_{t\in[0,T_0]}\|\bar{b}_{m}^{j+1}(t)\|^2_{L^2}\leq\lambda\sup_{t\in[0,T_0]}\|\bar{b}_{m}^{j}(t)\|^2_{L^2}
\end{equation}
for any $j\geq1$. This, together with (\ref{3.10}) implies that $(b^j_m)_{j\geq0}$ is a Cauchy sequence in the Banach space
$L^{\infty}([0,T_0];L^2(\Omega))$. Thus, taking into account the a priori estimate (\ref{3.8}), we infer that the limit function
\begin{equation*}
b_m=b_m^0+\lim_{n\longrightarrow\infty}\sum_{j=0}^n(b_m^{j+1}-b_m^{j})
\end{equation*}
indeed exists in $L_t^{\infty}L_x^2(Q_{T_0})\cap L_t^2H_x^1(Q_{T_0})$.

Furthermore, suppose that $b_m$ and $b_m'$ are two solutions in $L_t^{\infty}L_x^2(Q_{T_0})\cap L_t^2H_x^1(Q_{T_0})$. By the same process as in
(\ref{3.14}) to prove the convergence of $(b^j_m)_{j\geq0}$, we have
\begin{equation*}
\sup_{0\leq t\leq T_0}\|b_m(t)-b'_m(t)\|_{L^2}\leq\lambda_1\sup_{0\leq t\leq T_0}\|b_m(t)-b'_m(t)\|_{L^2},
\end{equation*}
with $\lambda_1\in(0,1)$, which implies $b_m=b_m'$ holds. This proves the uniqueness of weak solution $b_m$ to (\ref{3.4}).

In addition, one can easily prove that the solution $b_m$ to (\ref{3.4}) is continuously depends on initial-boundary data as well as the given velocity
field $\bar{u}_m$. Hence, the solution operator defined by (\ref{3.4}) $\Psi_{b_m}:C([0,T_0];V^m)\longrightarrow
L_t^{\infty}L_x^2(Q_{T_0})\cap L_t^2H_x^1(Q_{T_0})$, $\Psi_{b_m}: \bar{u}_m\longmapsto b_m$ is continuous.

\textbf{Step 2.} Existence and uniqueness for (\ref{3.5}). Once the solution $b_m$ is determined in (\ref{3.4}), now, we proceed to prove the existence of
$u_m(x,t)=\sum_{i=1}^mg_i^m(t)\xi_i(x)$ to (\ref{3.5}). Multiplying (\ref{3.5}) by $\xi_i(x)$, then we obtain a nonlinear ordinary equations for
$g_i^m(t)$. By the argument of ODE, we can derive the existence and uniqueness of local solution $g_i^m(t)$ on $[0,T_0']$ such that
\begin{equation}\label{3.014}
u_m=\sum_{i=1}^mg_i^m(t)\xi_i(x)\in H^1([0,T_0'];V^m),
\end{equation}
where $T_0'\in [0,T_0]$ may depend on $u_0$, $b_m$ and $m$.

Furthermore, by energy estimate of (\ref{3.5}), we infer that
\begin{align*}
\frac12\frac{d}{dt}\|u_m\|^2_{L^2}+\|\nabla u_m\|^2_{L^2}& \leq \|u_m\|_{L^{\infty}} \|b_m\|_{L^2}\|\nabla b_m\|_{L^2}\\
  &\leq c(m)\|u_m\|_{L^2}\|b_m\|_{L^2}\|\nabla b_m\|_{L^2} \\
&\leq c(m)\|b_m\|_{L^2}^2\|\nabla b_m\|_{L^2}^2+\frac12\|\nabla u_m\|_{L^2}^2,
\end{align*}
This, combined with (\ref{3.8})  implies that for all $t\in [0,T_0']$
\begin{align}\label{3.15}
&\sup_{t\in [0,T_0']}\|u_m(t)\|^2_{L^2}+\int_0^t\|\nabla u_m(\tau)\|^2_{L^2}d\tau\nonumber\\
&\leq \|u_{0m}\|^2_{L^2}+c\left[M^2(cM^qt+1)e^{cM^qt}\right]^2\|h\|^4_{L^2_tH_x^{\frac12}(\Gamma_t)}
+c_1(\|\partial_th\|_{L^2_tH_x^{-\frac12}(\Gamma_t)},\|h\|_{L^2_tH_x^{\frac12}(\Gamma_t)}),
\end{align}
where $c_1(\cdot)\searrow 0$ as $t\longrightarrow 0$.
Moreover, by the ODE arguments, we can prove that the unique local solution $u_m$ to (\ref{3.5}) continuously depends on its initial data and the given
function $b_m$. Hence, we conclude that the solution operator defined by (\ref{3.5})
$\Psi_{u_m}:L^{\infty}([0,T_0];L^2(\Omega))\cap L^2([0,T_0];H^1(\Omega))\longrightarrow H^1([0,T_0'];V^m)$,
$\Psi_{u_m}(b_m)=u_m$ is continuous.

\textbf{Step 3.} Existence and uniqueness for (\ref{3.1}). Set $T_1:=T_0'$, from the conclusion above, one can see that the mapping for all
$m\in \mathbb{N}$
\begin{equation*}
\Psi_{u_m}\circ\Psi_{b_m}:~C([0,T_1];V^m)\longrightarrow  H^1([0,T_1];V^m), \quad \Psi_{u_m}\circ\Psi_{b_m}(\bar{u}_m)=u_m
\end{equation*}
is continuous, where $u_m$ is the solution to (\ref{3.5}). By the Rellich theorem and the finite dimensionality of $V^m$, we infer that
$\Psi_{u_m}\circ\Psi_{b_m}$ is a compact from $C([0,T_1];V^m)$ into itself. Moreover, form (\ref{3.15}), it follows that
\begin{equation}\label{3.16}
\sup_{t\in [0,T_1]}\|u_m(t)\|^2_{L^2}\leq\frac M2+c\left[M^2(cM^qt+1)e^{cM^qt}\right]^2\|h\|^4_{L^2_tH_x^{\frac12}(\Gamma_t)}
+c_1(\|\partial_th\|_{L^2_tH_x^{-\frac12}(\Gamma_t)},\|h\|_{L^2_tH_x^{\frac12}(\Gamma_t)}),
\end{equation}
which easily yields that for suitable small $T_1$  there holds $\|u_m(t)\|^2_{L^2}\leq M$, for all $t\in [0,T_1]$. Hence, employing Schauder's fixed
point theorem, we conclude that there exists at least one fixed point $u_m$ to (\ref{3.1}) in the bounded closed convex set
\begin{equation*}
\left\{u_m\in C([0,T_1];V_m):~\sup_{t\in [0,T_1]}\|u_m(t)\|^2_{L^2}\leq M~\text{with}~u_m(0)=P_mu_0\right\}
\end{equation*}
such that (\ref{3.2}) holds. Finally, similar to
(\ref{3.12})-(\ref{3.14}), we can deduce the uniqueness of the approximate solutions $(u_m,b_m)$ to (\ref{3.8}). This completes the proof
of Proposition \ref{po3.1}.
\end{proof}
Now we give the definition of weak solutions to (\ref{3.1}).
\begin{Definition}\label{de3.1}
We say $(u_m,b_m)$ is a weak solution to (\ref{3.1}) on $Q_T$, if
\begin{equation*}
(u_m,b_m)\in L^{\infty}([0,T];H\times L^2(\Omega))\cap L^2([0,T]; V\times H^1(\Omega)),
\end{equation*}
and (\ref{3.1})$_1$-(\ref{3.1})$_2$ are valid in the weak sense.
\end{Definition}
As a consequence, from   Proposition \ref{po3.1} and Definition \ref{de3.1}, it follows that
\begin{equation*}
(u_m,b_m)\in  L^{\infty}([0,T_1];H\times L^2(\Omega))\cap L^2([0,T_1];V\times H^1(\Omega))
\end{equation*}
is a weak solution to (\ref{3.1}) on $Q_{T_1}$. Next, we propose to extend the time interval of the existence of weak solutions. Precisely, we have:
\begin{Lemma}\label{le3.1}
Let all assumptions in Theorem \ref{th1.1} be in force. Then for any $m>0$,
 the problem (\ref{3.1}) admits a unique weak solution $(u_m,b_m)$ on $Q_T$.
\end{Lemma}
\begin{proof}
First, we set $\tilde{b}_m:=b_m-h_E$, then (\ref{3.1}) can be rewritten as
\begin{equation}\label{3.17}
\begin{cases}
\partial_t u_m-\Delta u_m+u_m\cdot\nabla u_m-b_m\cdot\nabla b_m+\nabla p_m=0 &\text{in}~ Q_T,\\
\partial_t \tilde{b}_m-\Delta\tilde{b}_m+u_m\cdot\nabla b_m-b_m\cdot\nabla u_m+\partial_t h_E=0 &\text{in}~ Q_T,\\
\textrm{div}u_m=0&\text{in}~ Q_T,\\
u_m(0)=u_{0m},~\tilde{b}_m(x,0)=b_0-h_E(0)  &\text{in}~\Omega,\\
u_m=\tilde{b}_m(x,t)=0 &\text{on}~\Gamma_T.
\end{cases}
\end{equation}
We choose $u_m$ and $\tilde{b}_m$ as test functions in (\ref{3.17}), then by energy estimate of (\ref{3.17})$_1$--(\ref{3.17})$_2$, one has
\begin{align*}
\frac{1}{2}\frac{d}{dt}(\|u_m\|_{L^2}^{2}+\|\tilde{b}_m\|_{L^2}^{2})+\|\nabla u_m\|_{L^2}^{2}+\|\nabla\tilde{b}_m\|_{L^2}^{2}
=&\int_{\Omega}(b_m\cdot\nabla  b_m)\cdot u_mdx-\int_{\Omega}(u_m\cdot\nabla b_m)\cdot \tilde{b}_mdx\\
&+\int_{\Omega}(b_m\cdot\nabla u_m)\cdot \tilde{b}_mdx-\int_{\Omega}\partial_t h_E\cdot\tilde{b}_mdx\\
:= &I_1+I_2+I_3+I_4.
\end{align*}
By the H\"{o}lder, Young and Sobolev's inequalities, we obtain
\begin{align}\label{3.18}
I_1+I_3 &\leq\int_{\Omega}|((\tilde{b}_m+h_E)\cdot\nabla u_m)\cdot h_E| dx\nonumber\\
&\leq \|h_E\|_{L^{4}}\|\tilde{b}_m\|_{L^4}\|\nabla u_m\|_{L^2}+\|h_E\|_{L^{4}}^2\|\nabla u_m\|_{L^2}\nonumber\\
&\leq c\|h_E\|_{H^1}\|\tilde{b}_m\|_{L^2}^{\theta}\|\nabla\tilde{b}_m\|_{L^2}^{1-\theta}\|\nabla u_m\|_{L^2}
+\|h_E\|_{H^1}^2\|\nabla u_m\|_{L^2}\nonumber\\
&\leq c \|h_E\|_{H^1}^{q_n}\|\tilde{b}_m\|_{L^2}^2+\frac14\|\nabla\tilde{b}_m\|_{L^2}^2+\frac14\|\nabla u_m\|_{L^2}^2+2\|h_E\|_{H^1}^4,
\end{align}
where $\theta=\frac12$, $q_n=4$ for $n=2$ and  $\theta=\frac14$, $q_n=8$ for $n=3$.

Next, in virtue of Poincar\'{e}'s inequality, we further obtain
\begin{align}\label{3.19}
I_2+I_4 &\leq \|h_E\|_{L^{4}}\|u_m\|_{L^4}\|\nabla b_m\|_{L^2}+c\|\partial_t h_E\|_{L^2}\|\nabla\tilde{b}_m\|_{L^2}\nonumber\\
&\leq c\|h_E\|_{H^1}\|u_m\|_{L^2}^{\theta}\|\nabla u_m\|_{L^2}^{1-\theta}(\|\nabla \tilde{b}_m\|_{L^2}+\|\nabla h_E\|_{L^2})
+c\|\partial_t h_E\|_{L^2}^{2}+\frac{1}{8}\|\nabla\tilde{b}_m\|_{L^2}^{2}\nonumber\\
&\leq c\|h_E\|_{H^1}^{q_n}\|u_m\|_{L^2}^{2}+\|\nabla h_E\|_{L^2}^{2}+\frac14\|\nabla u_m\|_{L^2}^2
+\frac{1}{4}\|\nabla \tilde{b}_m\|_{L^2}^{2}+c\|\partial_t h_E\|_{L^2}^{2},
\end{align}
where $\theta$ and $q_n$ are determined in  (\ref{3.18}).

Putting these estimates together and taking into account Lemma \ref{le2.2}, there holds
\begin{align}\label{3.20}
&\frac{d}{dt}(\|u_m\|_{L^2}^{2}+\|\tilde{b}_m\|_{L^2}^{2})+\|\nabla u_m\|_{L^2}^{2}+\|\nabla\tilde{b}_m\|_{L^2}^{2}\nonumber\\
&\leq c\|h\|_{H^{\frac{1}{2}}(\Gamma)}^{q_n}(\|u_m\|_{L^2}^{2}+\|\tilde{b}_m\|_{L^2}^{2})+c(\|h\|_{H^{\frac{1}{2}}(\Gamma)}^{2}
+\|\partial_th\|_{H^{-\frac{1}{2}}(\Gamma)}^{2}+\|h\|_{H^{\frac{1}{2}}(\Gamma)}^{4}).
\end{align}
This, together with Gronwall's inequality and Proposition \ref{po3.1} implies that
\begin{align}\label{3.21}
\|u_m(t)\|_{L^2}^{2}+\|\tilde{b}_m(t)\|_{L^2}^{2}+\int_{0}^{t}(\|\nabla\tilde{b}_m\|_{L^2}^{2}
+\|\nabla u_m\|_{L^2}^{2})(\tau)d\tau&\leq(e^{\varphi(t)}\varphi(t)+1)\psi(t),
\end{align}
for all $t\in [0,T_1]$, where
\begin{align*}
\psi(t)=&\|u_0\|_{L^2}^{2}+\|b_0\|_{L^2}^{2}+c\int_{0}^{t}(\|h\|_{H^{\frac{1}{2}}(\Gamma)}^{2}
+\|\partial_th\|_{H^{-\frac{1}{2}}(\Gamma)}^{2}+\|h\|_{H^{\frac{1}{2}}(\Gamma)}^{4})d\tau\\
\varphi(t)=&c\int_{0}^{t}\|h(\tau)\|_{H^{\frac{1}{2}}(\Gamma)}^{q_n}d\tau.
\end{align*}
Now, we set
\begin{equation*}
M_T:=(e^{\varphi(T)}\varphi(T)+1)\psi(T)+c(\|h\|_{L^{2}([0,T];H^{\frac12}(\Gamma))},\|\partial_th\|_{L^{2}([0,T];H^{-\frac12}(\Gamma))})
\end{equation*}
for any $T\in (0,\infty)$. Choosing $M=2M_T$ in (\ref{3.3}), and $(u_m(T_1),b_m(T_1))$ as initial data. Similar to the proof of Proposition \ref{po3.1}
and (\ref{3.21}), we conclude that there exists a constant $\delta>0$ depends on $m$, $\Omega$, $M_T$ such that (\ref{3.1}) has a unique weak solution
$(u_m,b_m)$ on $\Omega\times [T_1,T_2]$ with $T_2=T_1+\delta$. Furthermore, $(u_m,b_m)$ satisfies
\begin{align}\label{3.021}
&\sup_{0\leq t\leq T_2}(\|u_m(t)\|_{L^2}^{2}+\|b_m(t)\|_{L^2}^{2})+\int_0^t(\|\nabla u_m(\tau)\|_{L^2}^{2}+\|\nabla b_m(\tau)\|_{L^2}^{2})d\tau
\nonumber\\
&\leq (e^{\varphi(t)}\varphi(t)+1)\psi(t)+c(\|h\|_{L^{2}_tH_x^{\frac12}(\Gamma_t)},\|\partial_th\|_{L^{2}_tH_x^{-\frac12}(\Gamma_t)})\leq M_T,
\end{align}
for all $t\in [0,T_2]$. At this stage, if $T_2=T$, we have completed the proof. If $T_2<T$, we can repeat the process as before. After iterating
$[\frac{T-T_1}{\delta}]+1$ times, then we conclude  that
\begin{equation*}
(\|u_m(t)\|_{L^2}^{2}+\|b_m(t)\|_{L^2}^{2})+\int_0^t(\|\nabla u_m(\tau)\|_{L^2}^{2}+\|\nabla b_m(\tau)\|_{L^2}^{2})d\tau\leq M_T
\end{equation*}
holds for all $t\in [0,T]$. This completes the proof of Lemma \ref{le3.1}.
\end{proof}
Based on Lemma \ref{le3.1}, now, we are able to prove Theorem \ref{th1.1}.
\begin{proof}[Proof of Theorem \ref{th1.1}]
Let $(u_m,b_m)$ be a sequence of weak solutions to (\ref{3.1}) on $Q_T$, and $m=1,2,\cdots$. First, it is obvious that for all $v\in V$
\begin{align*}
\int_{\Omega}\partial_t u_m\cdot vdx&\leq\left|\int_{\Omega}(-u_m\cdot\nabla u_m+\Delta u_m+b_m\cdot\nabla b_m)\cdot vdx\right| \\
& \leq (\|u_m\|^{2}_{L^4}+\|\nabla u_m\|_{L^2}+\|b_m\|^{2}_{L^4})\|\nabla v\|_{L^2}\\
& \leq c(\|u_m\|^{2\theta}_{L^2}\|\nabla u_m\|^{2(1-\theta)}_{L^2}+\|\nabla u_m\|_{L^2}+\|b_m\|^{2\theta}_{L^2}\|b_m\|^{2(1-\theta)}_{H^1})
\|\nabla v\|_{L^2},
\end{align*}
where $\theta=\frac12$ for n=2 and $\theta=\frac14$ for $n=3$. This implies that $\partial_tu_m\in L^p([0,T];V')$ with
$p=2$ for $n=2$ and $p=\frac43$ for $n=3$. By the same way, we can also prove that $\partial_t\tilde{b}_m\in L^p([0,T];H^{-1})$. Thus,
by the Aubin-Lions Theorem, we conclude that the sequence $(u_m,b_m)_{m=1}^{\infty}$ is pre-compact in $L^2([0,T];H\times L^2)$. Furthermore, by
extracting subsequence (if necessary), we can see that there is $(u,b)$ such that
\begin{align}
&(u_m,b_m)\longrightarrow (u,b) \quad \text{weakly in}~L^2([0,T];V\times H^1),\label{3.22} \\
&(u_m,b_m)\longrightarrow (u,b) \quad \text{weak-star in}~L^{\infty}([0,T];H\times L^2),\label{3.23} \\
&(u_m,b_m)\longrightarrow (u,b) \quad \text{strongly in}~L^2([0,T];H\times L^2).\label{3.24}
\end{align}
Finally, we propose to show that $(u,b)$ is a weak solution to (\ref{1.2})-(\ref{1.3}). In fact, for all $\varphi\in C^{\infty}([0,T];V)$ with
$\varphi(T)=0$, there holds
\begin{equation*}
-\int_0^T\int_{\Omega}u_m\cdot\partial_t\varphi dxdt
+\int_0^T\int_{\Omega}[\nabla u_m-(u_m\otimes u_m)+(b_m\otimes b_m)]\cdot \nabla\varphi dxdt
=\int_{\Omega} u_{0m}\cdot \varphi(0,x)dx.
\end{equation*}
By (\ref{3.22})-(\ref{3.24}) and $u_{0m}\longrightarrow u_0$ in $H$, it is easy to see that
\begin{equation*}
-\int_0^T\int_{\Omega}u\cdot\partial_t\varphi dxdt
+\int_0^T\int_{\Omega}[\nabla u-(u\otimes u)+(b\otimes b)]\cdot \nabla\varphi dxdt
=\int_{\Omega} u_{0}\cdot \varphi(0,x)dx.
\end{equation*}
as $m\longrightarrow\infty$.  This implies that $u$ is a weak solution to (\ref{1.3}). By the same way, we can also conclude that $b$ is a weak
solution to (\ref{1.3}). In addition, we can also prove that $(u,b)$ satisfies the initial-boundary conditions (\ref{1.2})-(\ref{1.02}).
Finally, similar to (\ref{3.21})-(\ref{3.021}), for the weak solutions $(u,b)$, we have
\begin{equation*}
\sup_{0\leq t\leq T}(\|u(t)\|_{L^2}^{2}+\|b(t)\|_{L^2}^{2})+\int_0^t(\|\nabla u(\tau)\|_{L^2}^{2}+\|\nabla b(\tau)\|_{L^2}^{2})d\tau
\leq M_T.
\end{equation*}
Additionally, if $n=2$, the uniqueness for weak solutions $(u,b)$ follows from Theorem \ref{th3.1} below.
Thus, we have completed the proof of  Theorem \ref{th1.1}.
\end{proof}
Having proved the existence of weak solutions, if the spatial dimension $n=2$, we can further prove the continuous dependence on initial-boundary
conditions, which easily yields the uniqueness of global weak solution.
\begin{Theorem}\label{th3.1}
(Continuous dependence in the $2D$ case.) Let all  assumptions of theorem \ref{th1.1} be in force. if $n = 2$, then the problem (\ref{1.2})-(\ref{1.3})
admits a unique weak solution. Moreover, let $(u^{(i)},b^{(i)})$ $(i=1,2)$ be two weak solutions to (\ref{1.3}) corresponding to the initial data
$(u_0^{(i)},b_0^{(i)})$ and boundary data $(0,h^{(i)})$. Denoting $\bar{u}=u^{(1)}-u^{(2)}$,  $\bar{b}=b^{(1)}-b^{(2)}$, $\bar{u}_0=u_0^{(1)}-u_0^{(2)}$,
$\bar{b}_0=b_0^{(1)}-b_0^{(2)}$ and $\bar{h}=h^{(1)}-h^{(2)}$, then  the following estimate holds:
\begin{equation}\label{3.26}
\|\bar{ u}\|_{L^2}^{2}+\|\bar{b}\|_{L^2}^{2}+\int_{0}^{t}(\|\nabla\bar{ u}(\tau)\|_{L^2}^{2}
+\|\nabla\bar{b}(\tau)\|_{L^2}^{2})d\tau\leq c\left[\|\bar{u}_0\|_{L^2}^{2}+\|\bar{ b}_0\|_{L^2}^{2}
+c\int_{0}^{t}\|\bar{ h}(\tau)\|_{H^{\frac{1}{2}}(\Gamma)}^{2}d\tau\right],
\end{equation}
where $c$ depends on  $T$, $\Omega$, $\|u_0^{(i)}\|_{L^2}$, $\|b_0^{(i)}\|_{L^2}$, $\|h^{(i)}\|_{L^{4}([0,T];H^{\frac12})}$,
$\|\partial_th^{(i)}\|_{L^{2}([0,T];H^{-\frac12})}$.
\end{Theorem}
\begin{proof}
By considering the difference of the equations solved by $(u^{(1)},b^{(1)})$, $(u^{(2)},b^{(2)})$, we have:
\begin{equation}\label{3.27}
\begin{cases}
\partial_t\bar{u}- \Delta\bar{u}+\bar{u}\cdot\nabla u^{(1)}+u^{(2)}\cdot\nabla \bar{u}-\bar{b}\cdot\nabla b^{(1)}-b^{(2)}\cdot\nabla \bar{b}
+\nabla\bar{p}=0,\\
\partial_t\bar{b}- \Delta\bar{b}+\bar{u}\cdot\nabla b^{(1)}+u^{(2)}\cdot\nabla \bar{b}-\bar{b}\cdot\nabla u^{(1)}-b^{(2)}\cdot\nabla \bar{u}=0,
\end{cases}
\end{equation}
where $\bar{p}=p^{(1)}-p^{(2)}$ and $p^{(i)}$ is the pressure terms corresponding to $(u^{(i)},b^{(i)})$.
Multiplying  $(\ref{3.27})_1$ and $(\ref{3.27})_2$ with $\bar{u}$  and  $\bar{b}$, respectively, we obtain
\begin{align*}
\frac{1}{2}&\frac{d}{dt}(\|\bar{u}\|_{L^2}^{2}+\|\bar{b}\|_{L^2}^{2})
+\|\nabla\bar{u}\|_{L^2}^{2}+\|\nabla\bar{b}\|_{L^2}^{2}\\
=&\int_{\Omega}(\bar{b}\cdot\nabla b^{(1)}+ b^{(2)}\cdot\nabla\bar{b}-\bar{u}\cdot\nabla u^{(1)})\cdot\bar{u}dx \\
&+\int_{\Omega}(\bar{b}\cdot\nabla u^{(1)}
+ b^{(2)}\cdot\nabla \bar{u}-\bar{u}\cdot\nabla b^{(1)})\cdot\bar{b}dx
+\int_{\Gamma}\partial_{\nu}\bar{b}\cdot \bar{h}ds,\\
:=&K_1+K_2+K_3.
\end{align*}
Applying the H\"{o}lder, Young and Sobolev's inequalities, it holds that
\begin{align*}
K_1+K_2+K_3 \leq&(\|\bar{u}\|_{L^4}^{2}+\|\bar{b}\|_{L^4}^{2})\|\nabla u^{(1)}\|_{L^2}\\
&+2\|\bar{b}\|_{L^4}\|\bar{u}\|_{L^4}\|\nabla b^{(1)}\|_{L^2}
+\|\partial_{\nu}\bar{b}\|_{H^{-\frac12}(\Gamma)}\|\bar{h}\|_{H^{\frac12}(\Gamma)}\\
\leq& c(\|\bar{u}\|_{L^2}\|\nabla\bar{u}\|_{L^2}+\|\bar{b}\|_{L^2}\|\bar{b}\|_{H^1})\|\nabla u^{(1)}\|_{L^2}\\
&+c\|\bar{b}\|_{L^2}^{\frac{1}{2}}\|\bar{b}\|_{H^1}^{\frac{1}{2}}\|\bar{u}\|_{L^2}^{\frac{1}{2}}
\|\nabla\bar{u}\|_{L^2}^{\frac{1}{2}}\|\nabla b^{(1)}\|_{L^2}
+\frac{1}{4}\|\nabla\bar{b}\|_{L^2}^2+c\|\bar{h}\|^2_{H^{\frac12}}.
\end{align*}
Note that
\begin{equation*}
\|\bar{b}\|_{H^1}\leq\|\tilde{\bar{b}}\|_{H^1}+ \|\bar{h}_E\|_{H^1}
\leq c\|\nabla\tilde{\bar{b}}\|_{L^2}+c\|\bar{ h}\|_{H^{\frac{1}{2}}(\Gamma)}
\leq c\|\nabla\bar{b}\|_{L^2}+c\|\bar{ h}\|_{H^{\frac{1}{2}}(\Gamma)}
\end{equation*}
with $\tilde{\bar{b}}=\bar{b}-\bar{h}_E$.

Putting these estimates together, we arrive at
\begin{align*}
&\frac{d}{dt}(\|\bar{u}\|_{L^2}^{2}+\|\bar{b}\|_{L^2}^{2})+\|\nabla\bar{u}\|_{L^2}^{2}+\|\nabla\bar{b}\|_{L^2}^{2}\\
&\leq c(\|\nabla u^{(1)}\|_{L^2}^{2}+\|\nabla b^{(1)}\|_{L^2}^{2})(\|\bar{u}\|_{L^2}^{2}+\|\bar{b}\|_{L^2}^{2})
+c\| \bar{h}\|_{H^{\frac{1}{2}}(\Gamma)}^{2}.
\end{align*}
This, together with Gronwall's inequality  implies  (\ref{3.26}). Thus, we have completed the proof of Theorem \ref{th3.1}.
\end{proof}
\subsection{Existence and uniqueness of strong solution}
In this section, when $n=2$,  we aim to prove the existence of strong solution to (\ref{1.2})-(\ref{1.3}) under the more regular initial-boundary
conditions (\ref{1.8}). First, we introduce the definition of strong solution to (\ref{1.3}).
\begin{Definition}\label{de3.2}
We say that a pair $(u,b)$ is a strong solution to the problem (\ref{1.2})-(\ref{1.3}), if
\begin{itemize}
\item  it is a weak solution and moreover $(u,b)\in L^{\infty}([0,T];V\times H^1)\cap L^2([0,T];H^2\times H^2)$.
\item  $(u_0,b_0)\in V\times H^1$ and the equation (\ref{1.3}) holds almost everywhere on $Q_T$.
\end{itemize}
\end{Definition}
Now, we start to prove Theorem \ref{th1.2}.
\begin{proof}[Proof of Theorem \ref{th1.2}]
Multiplying  $(\ref{1.3})_1$ and $(\ref{1.3})_2$ with $Su=-\Delta u+\nabla p$ and $-\Delta \hat{b}$, respectively, we have
\begin{align*}
&\frac{1}{2}\frac{d}{dt}(\|\nabla u\|_{L^2}^{2}+\|\nabla\hat{b}\|_{L^2}^{2})
+\|Su\|_{L^2}^{2}+\|\Delta\hat{b}\|_{L^2}^{2}\\
&= \int_{\Omega}(u\cdot\nabla u)\cdot Sudx-\int_{\Omega}(b\cdot\nabla b)\cdot Su dx
+\int_{\Omega}(u\cdot\nabla b)\cdot\Delta\hat{b}dx-\int_{\Omega}(b\cdot\nabla u)\cdot\Delta\hat{b}dx \\
&:= J_1+J_2+J_3+J_4,
\end{align*}
where $\hat{b}=b-h_p$.

By Lemma \ref{le2.1}, Young's inequality and Sobolev's inequality, we can see that
\begin{align*}
J_1&\leq\|u\|_{L^{\infty}}\|\nabla u\|_{L^2}\|Su\|_{L^2}\leq c\|u\|_{L^2}^{\frac{1}{2}}\|u\|_{H^2}^{\frac{1}{2}}\|\nabla u\|_{L^2}\|Su\|_{L^2}\\
&\leq c\|u\|_{L^2}^{2}\|\nabla u\|_{L^2}^{4}+\frac18\|Su\|_{L^2}^{2}.
\end{align*}
Next, by Sobolev's inequality and the equivalent norms $\|\hat{b}\|_{H^2}\approx \|\Delta\hat{b}\|_{L^2}$ (cf. \cite{bo4}),
we are in a position to obtain
\begin{align*}
J_2&\leq\|b\|_{L^{\infty}}\|\nabla b\|_{L^2}\|Su\|_{L^2}  \\
&\leq c\|b\|_{L^2}^{\frac{1}{2}}\|b\|_{H^2}^{\frac{1}{2}}(\|\nabla\hat{b}\|_{L^2}+\|\nabla h_p\|_{L^2})\|Su\|_{L^2}\\
&\leq\frac18\|Su\|_{L^2}^{2}+\frac14\|\Delta\hat{b}\|_{L^2}^{2}+c\|{h}\|_{H^{\frac{3}{2}}(\Gamma)}^{2}
+c\|b\|_{L^2}^{2}(\|\nabla \hat{b}\|_{L^2}^{4}+\|h\|_{H^{\frac{1}{2}}(\Gamma)}^{4}).
\end{align*}
Similarly, we further obtain
\begin{align*}
J_3+J_4&\leq\|u\|_{L^{\infty}}\|\nabla b\|_{L^2}\|\Delta\hat{b}\|_{L^2}+\|b\|_{L^{\infty}}\|\nabla u\|_{L^2}
\|\Delta\hat{b}\|_{L^2}\\
&\leq c\|u\|_{L^2}^{\frac{1}{2}}\| u\|_{H^2}^{\frac{1}{2}}\|\nabla b\|_{L^2}\|\Delta\hat{b}\|_{L^2}
+c\|b\|_{L^2}^{\frac{1}{2}}\|b\|_{H^2}^{\frac{1}{2}}\|\nabla u\|_{L^2}\|\Delta\hat{b}\|_{L^2}\\
&\leq\frac18\|Su\|_{L^2}^{2}+\frac14\|\Delta\hat{b}\|_{L^2}^{2}
+c\|u\|_{L^2}^{2}(\|\nabla\hat{b}\|_{L^2}^{4}+\|h\|_{H^{\frac{1}{2}}(\Gamma)}^{4})
+c\|{h}\|_{H^{\frac{3}{2}}(\Gamma)}^{2}+c\|b\|_{L^2}^{2}\|\nabla u\|_{L^2}^{4}.
\end{align*}
Putting these estimates together, then we have
\begin{align}\label{3.28}
&\frac{d}{dt}(\|\nabla u\|_{L^2}^{2}+\|\nabla\hat{b}\|_{L^2}^{2})+\|Su\|_{L^2}^{2}+\|\Delta\hat{b}\|_{L^2}^{2}\nonumber\\
&\leq K(\|\nabla u\|_{L^2}^{2}+\|\nabla\hat{b}\|_{L^2}^{2})+c(\|u\|_{L^2}^{2}+\|b\|_{L^2}^{2})\|h\|_{H^{\frac{1}{2}}(\Gamma)}^{4}
+c\|h\|_{H^{\frac{3}{2}}(\Gamma)}^{2},
\end{align}
where
\begin{equation*}
K:=c(\|u\|_{L^2}^{2}+\|b\|_{L^2}^{2})(\|\nabla u\|_{L^2}^{2}+\|\nabla\hat{b}\|_{L^2}^{2}).
\end{equation*}
Applying Gronwall's inequality,  from (\ref{3.28}), it follows that
\begin{equation*}
\|\nabla u\|_{L^2}^{2}+\|\nabla\hat{b}\|_{L^2}^{2}+\int_{0}^{t}(\|Su\|_{L^2}^{2}+\|\Delta\hat{b}\|_{L^2}^{2})(\tau)d\tau
\leq\phi(t)e^{\phi(t)}\omega(t)+\omega(t),
\end{equation*}
where
\begin{align*}
\phi(t)=&e^{\int_{0}^{t}K(\tau)d\tau},\\
\omega(t)=&\|\nabla u_0\|_{L^2}^{2}+\|\nabla\hat{b}_0\|_{L^2}^{2} +c\int_{0}^{t}(\|u(\tau)\|_{L^2}^{2}+\|b(\tau)\|_{L^2}^{2})\|h(\tau)\|_{H^{\frac{1}{2}}(\Gamma)}^{4}d\tau
+c\int_{0}^{t}\|h(\tau)\|_{H^{\frac{3}{2}}(\Gamma)}^{2}d\tau.
\end{align*}
This, combined  with  (\ref{2.4}) implies (\ref{1.9}). Finally, the uniqueness of strong solutions can be derived from  Theorem \ref{th3.2}.
Thus, we have completed the proof of Theorem \ref{th1.2}.
\end{proof}
Analogous to Theorem \ref{th3.1}, based on the existence of strong solution $(u,b)$ in Theorem \ref{th1.2}, now, we proceed to prove the continuous dependence of initial-boundary data, from which, we derive the uniqueness of strong solutions $(u,b)$.
\begin{Theorem}\label{th3.2}
(Continuous dependence in the $2D$ case.) Let all assumptions of theorem \ref{th1.2} be verified. Then the problem (\ref{1.2})-(\ref{1.3}) admits
a unique strong solution. Moreover, let $(u^{(i)},b^{(i)})(i=1,2)$ be two strong solutions to (\ref{1.3}) corresponding to the initial data
$(u_{0}^{(i)},b_{0}^{(i)})$ and boundary data $(0,h^{(i)})$.  then  the following estimate holds:
\begin{align}\label{3.29}
&\|\nabla\bar{u}\|_{L^{2}}^{2}+\|\bar{b}\|_{H^1}^{2}+\int_{0}^{t}(\|\bar{ u}(\tau)\|_{H^2}^{2}+\|\bar{b}(\tau)\|_{H^2}^{2})d\tau\nonumber\\
&\leq c\left[\|\nabla\bar{u}_0\|_{L^2}^{2}+\|\bar{b}_0\|_{H^1}^{2}
+c\int_{0}^{t}(\|\bar{h}(\tau)\|_{H^{\frac{3}{2}}(\Gamma)}^{2}
+\|\partial_t\bar{h}(\tau)\|_{H^{-\frac{1}{2}}(\Gamma)}^{2})d\tau\right],
\end{align}
where $c$ be a positive constant depends on $T$, $\Omega$, $\|\nabla u_0^{(i)}\|_{L^2}$, $\|b_0^{(i)}\|_{H^1}$, $\|h^{(i)}\|_{L^{2}([0,T];H^{\frac32})}$,
$\|\partial_th^{(i)}\|_{L^{2}([0,T];H^{-\frac12})}$.
\end{Theorem}
\begin{proof}
The process  is similar with Theorem \ref{th3.1}, here, we just give a sketch of the proof. Multiplying (\ref{3.27})$_1$ and (\ref{3.27})$_2$ with
$S\bar{u}=-\Delta \bar{u}+\nabla \bar{p}$ and $-\Delta \hat{\bar{b}}$, respectively, we have
\begin{align}\label{3.30}
\frac{1}{2}&\frac{d}{dt}(\|\nabla\bar{u}\|_{L^{2}}^{2}+\|\nabla\hat{\bar{b}}\|_{L^{2}}^{2})
+\|S\bar{u}\|_{L^{2}}^{2}+\|\Delta\hat{\bar{b}}\|_{L^{2}}^{2}\nonumber\\
=&\int_{\Omega}(\bar{b}\cdot\nabla b^{(1)}+b^{(2)}\cdot\nabla\bar{b})\cdot S\bar{u}dx
-\int_{\Omega}(\bar{u}\cdot\nabla u^{(1)}+u^{(2)}\cdot\nabla\bar{u})\cdot S\bar{u}dx\nonumber\\
&-\int_{\Omega}(\bar{b}\cdot\nabla u^{(1)}+b^{(2)}\cdot\nabla\bar{u})\cdot \Delta\hat{\bar{b}}dx
+\int_{\Omega}(\bar{u}\cdot\nabla b^{(1)}+u^{(2)}\cdot\nabla\bar{b})\cdot\Delta\hat{\bar{b}}dx\nonumber\\
:=&W_{1}+W_{2}+W_{3}+W_{4},
\end{align}
where $\hat{\bar{b}}=\bar{b}-\bar{h}_p$, $\bar{h}_p=h_p^{(1)}-h_p^{(2)}$ and $h^{(i)}_p$ be the lifting functions of $h^{(i)}$.

For the term $W_{1}$, by H\"{o}lder, Young and Sobolev's inequalities, we deduce that
\begin{align*}
W_1&=\int_{\Omega}[(\hat{\bar{b}}+\bar{h}_p)\cdot\nabla b^{(1)}+b^{(2)}\cdot\nabla(\hat{\bar{b}}+\bar{h}_p)]\cdot S\bar{u}dx  \\
& \leq [(\|\hat{\bar{b}}\|_{L^4}+\|\bar{h}_p\|_{L^4})\cdot\|\nabla b^{(1)}\|_{L^4}
+\|b^{(2)}\|_{L^{\infty}}\cdot(\|\nabla\hat{\bar{b}}\|_{L^2}+\|\nabla\bar{h}_p\|_{L^2})]\cdot \|S\bar{u}\|_{L^2}\\
& \leq c(\|b^{(1)}\|_{H^2}^2+\|b^{(2)}\|_{H^2}^2)\|\nabla\hat{\bar{b}} \|_{L^2}^2
+c(\|b^{(1)}\|_{H^2}^2+\|b^{(2)}\|_{H^2}^2)\|h_p\|_{H^1}^2+\frac14\|S\bar{u}\|_{L^2}^2.
\end{align*}
Analogously, we further obtain
\begin{align*}
W_2&\leq c(\|u^{(1)}\|_{H^2}^2+\|u^{(2)}\|_{H^2}^2)\|\nabla\bar{u} \|_{L^2}^2 +\frac14\|S\bar{u}\|_{L^2}^2, \\
W_3&\leq c\|u^{(1)}\|_{H^2}^2 (\|\nabla \hat{\bar{b}}\|_{L^2}^2+\|\bar{h}_p\|_{H^1}^2)+c\|b^{(2)}\|_{H^2}\|\nabla\bar{u} \|_{L^2}^2
 +\frac14\|\Delta \hat{\bar{b}}\|_{L^2}^2, \\
W_4&\leq c\|b^{(1)}\|_{H^2}^2\|\nabla\bar{u} \|_{L^2}^2+c\|u^{(2)}\|_{H^2}^2(\|\nabla \hat{\bar{b}}\|_{L^2}^2+\|\nabla\bar{h}_p\|_{L^2}^2)
 +\frac14\|\Delta \hat{\bar{b}}\|_{L^2}^2.
\end{align*}
Inserting these estimates into (\ref{3.30}), yields that
\begin{align*}
\frac{d}{dt}&(\|\nabla\bar{u}\|_{L^{2}}^{2}+\|\nabla\hat{\bar{b}}\|_{L^{2}}^{2})
+\|S\bar{u}\|_{L^{2}}^{2}+\|\Delta\hat{\bar{b}}\|_{L^{2}}^{2}\nonumber\\
\leq&c(\|u^{(1)}\|_{H^2}^2+\|u^{(2)}\|_{H^2}^2+\|b^{(1)}\|_{H^2}^2+\|b^{(2)}\|_{H^2}^2)
(\|\nabla\bar{u} \|_{L^2}^2+\|\nabla \hat{\bar{b}}\|_{L^2}^2)\\
&+c(\|u^{(1)}\|_{H^2}^2+\|u^{(2)}\|_{H^2}^2+\|b^{(1)}\|_{H^2}^2+\|b^{(2)}\|_{H^2}^2)\|\bar{h}_p\|_{H^1}^2.
\end{align*}
This, together with \ref{th1.2}, Lemma \ref{le2.3} and Gronwall's inequality implies (\ref{3.29}).
Thus, we have completed the proof of Theorem \ref{th3.2}.
\end{proof}
\section{Uniform attractors in the two-dimensional case}\label{se4}
In this section, we aim to study the existence of a uniform attractor for (\ref{1.2})-(\ref{1.3}) with $n=2$. We suppose that the time
dependency can be completely described by a finite set of functions, and we denote it by $\sigma(t)$. In particular, in what follows, we call $\sigma(t)$
the (time) symbol and the set of all symbols will be called symbol space, which will usually be denoted by $\Sigma$. Then we give some fundamental
definition (see e.g. \cite{vm4}).
\begin{Definition}\label{de4.1}
Let $\Sigma$ be a symbol space. $\{U_{\sigma}(t,\tau),t\geq\tau,\tau\in\mathbb{R}\},\sigma\in\Sigma$ is said to be a family of processes in Banach space X, if the two-parameter family of mappings $\{U_{\sigma}(t,\tau)\}$ from $X$ to $X$ satisfy:
\begin{align*}
U_{\sigma}(t,s)\circ U_{\sigma}(s,\tau)&=\{U_{\sigma}(t,\tau)\},\forall\ t\geq s\geq\tau,\tau\in\mathbb{R},\\
U_{\sigma}(\tau,\tau)&=Id~ \text{(the identity operator)},~ \tau\in\mathbb{R}.
\end{align*}
where $\Sigma$ is a symbol space and $\sigma\in\Sigma$ is a symbol.
\end{Definition}
\begin{Definition}\label{de4.2}
We call set  $B_0\subset X$ the  uniformly (with respect to $\sigma\in\Sigma$) absorbing set for the family of process  $\{U_{\sigma}(t,\tau)\},$
$\sigma\in\Sigma$ if for any $\tau\in\mathbb{R}$ and every $B\in\mathcal{B}(X)$ there exists an absorbtion time $T_0=T_0(\tau,B)\geq\tau$ such that
$\cup_{\sigma\in\Sigma}U_{\sigma}(t,\tau)B\subset B_0$ for all $t\geq T_0$.
\end{Definition}
\begin{Definition}\label{de4.3}
A set $E\subset X$ is said to be  uniformly (w.r.t. $\sigma\in\Sigma$) attracting for the family of processes $\{U_{\sigma}(t,\tau)\},$
$\sigma\in\Sigma$ if for any fixed $\tau\in\mathbb{R}$ and every $B\in\mathcal{B}(X)$, there holds
\begin{equation*}
\lim\limits_{t \to \infty}\displaystyle{\sup_{\sigma\in\Sigma}}\ dist_{X}(U_{\sigma}(t,\tau)B,E)=0.
\end{equation*}
Here $dist_{X}(\cdot,\cdot)$ denotes  the Hausdorff semi-distance between subsets of a metric space $(X,d_{X})$.
\end{Definition}
\begin{Definition}\label{de4.4}
A closed set $\mathcal{A}_{\Sigma}\subset X$ is said to be the uniformly (w.r.t. $\sigma\in\Sigma$) attractor for the family of processes $\{U_{\sigma}(t,\tau)\},\sigma\in\Sigma$ if $\mathcal{A}_{\Sigma}$ satisfies the attracting property and the minimality property, namely
\begin{description}
\item[(i)] $\mathcal{A}_{\Sigma}$ is uniformly (w.r.t. $\sigma\in\Sigma$) attracting set;
\item[(ii)] $\mathcal{A}_{\Sigma}$ is contained in any closed uniformly attracting set.
\end{description}
\end{Definition}
In order to prove the existence of a uniform attractor for (\ref{1.2})-(\ref{1.3}), we will use the following additional definition.
\begin{Definition}\label{de4.5}
A family of processes $\{U_{\sigma}(t,\tau)\},\sigma\in\Sigma$ is said to be uniformly (w.r.t. $\sigma\in\Sigma$) $\omega$-limit compact if for any $\tau\in\mathbb{R}$ and any set $B\in\mathcal{B}(X)$, there holds
\begin{equation*}
B_t=\bigcup\limits_{\sigma\in\Sigma}\bigcup\limits_{s\geq t}U_{\sigma}(s,\tau)B
\end{equation*}
is bounded for all $t$ and $\lim\limits_{t \to \infty}\alpha(B_t)=0$. Here $\alpha$ is the Kuratowski measure, defined by
\begin{equation*}
\alpha(B):=inf\{r>0:~\text{$B$ has a finite cover by sets of $X$ with diameter less than $r$}\}.
\end{equation*}
\end{Definition}
In addition, for  present the main results we will use to prove the existence of a uniform attractor for (\ref{1.2})-(\ref{1.3}), we shall need the
following hypotheses:
\begin{description}
\item[(a1)] Let $\{T(t):~t\geq0\}$ be a family of operators acting on $\Sigma$ and satisfy
\begin{itemize}
\item $\{T(t)\}$ be a weakly continuous invariant semigroup on $\Sigma$, $T(t)\Sigma=\Sigma, \forall\ t\in\mathbb{R}_{+}$;
\item translation identity: $U_{\sigma}(t+s,\tau+s)=U_{T(s)\sigma}(t,\tau),\forall\ \sigma\in\Sigma,t\geq\tau,\tau\in\mathbb{R},s\geq0$.
\end{itemize}
\item[(a2)] Let $\Sigma$ be a weakly compact subset of some Banach space and $\{U_{\sigma}(t,\tau)\},\sigma\in\Sigma$ be $(X\times\Sigma,X)$-weakly continuous family of processes acting in $X$.
\end{description}
The following results  we will use in this section to prove the existence of a uniform attractor for (\ref{1.2})-(\ref{1.3}), and we
formulate it in the form we need (cf. \cite{bo21}).
\begin{Theorem}\label{th4.1}
Let the hypotheses (a1)-(a2) be verified. Suppose $\{U_{\sigma}(t,\tau)\}$, $\sigma\in \Sigma$ be a uniformly (w.r.t. $\sigma\in\Sigma$) $\omega$-limit
compact process in $X$ and has a weakly compact uniformly (w.r.t. $\sigma\in\Sigma$) absorbing set $B_0$. Then it possesses compact uniform
(w.r.t. $\sigma\in\Sigma$) attractor $\mathcal{A}_{\Sigma}$ satisfying
\begin{equation*}
\mathcal{A}_{\Sigma}=\omega_{0,\Sigma}(B_0)=\bigcup\limits_{\sigma\in\Sigma}\mathcal{K}_{\sigma}(0),~\forall s\in \mathbb{R}.
\end{equation*}
Here $\mathcal{K}_{\sigma}(s)$ is the section at $t=s$ of kernel $\mathcal{K}_{\sigma}$ of the process $\{U_{\sigma}(t,\tau)\}$ with symbol
$\sigma\in\Sigma$:
\begin{equation*}
\mathcal{K}_{\sigma}(s)=\left\{u(s):~u~\text{is a bounded complete trajectory of the  processes}~ U_{\sigma}(t,\tau)\right\}.
\end{equation*}
Furthermore, $\mathcal{K}_{\sigma}(s)$ is nonempty for all $\sigma\in\Sigma$.
\end{Theorem}
Next, we introduce a useful conclusion which will be used to prove the uniform $\omega$-limit compact for a given process. Its proof can be retrieved e.g. from \cite{bo21}.
\begin{Lemma}\label{le4.1}
Let $X$ be a  uniform convex Banach space.  If for any fixed $\tau\in\mathbb{R}$, $B\in\mathcal{B}(X)$ and $\varepsilon>0$, there exists $T_0=T_0(\tau,B,\varepsilon)\geq\tau$ and a finite dimensional subspace $X_1$ of $X$ such that
\begin{description}
  \item[($i_1$)]$P(\cup_{\sigma\in\Sigma}\cup_{t\geq T_0}U_{\sigma}(t,\tau)B)$ is bounded
  \item[($i_2$)]$ \|(Id-P)(\cup_{\sigma\in\Sigma}\cup_{t\geq T_0}U_{\sigma}(t,\tau)u\|_{X}\leq\varepsilon$, $\forall\ u\in B$,
\end{description}
where $P:X\rightarrow X_1$ is a bounded projector. Then the family of processes $\{U_{\sigma}(t,\tau)\},\sigma\in\Sigma$ is uniformly
(w.r.t. $\sigma\in\Sigma)$ $\omega$-limit compact,
\end{Lemma}
\subsection{Bounded absorbing sets for (\ref{1.2})-(\ref{1.3})}\label{se4.1}
For applying the Lemma \ref{le4.1} to prove the existence of uniform abstractor, we need to obtain some absorbing sets for the trajectories of
(\ref{1.2})-(\ref{1.3}). The symbol spaces in our cases is generated by the boundary  data $h(x,t)$. Before introducing the symbol spaces, we first recall
the definition of normal function spaces (see, e.g. \cite{bo21}).
\begin{Definition}\label{de4.6}
Let $E$ be a reflexive separable Banach space. We call a function $g\in L_{loc}^{p}(\mathbb{R},E)$ ($1\leq p<\infty$) is normal if for every
$\varepsilon>0$, there exists $\eta>0$ such that:
\begin{equation*}
\displaystyle{\sup_{t\in\mathbb{R}}}\int_{t}^{t+\eta}\|g(\tau)\|_{E}^{p}d\tau\leq\varepsilon.
\end{equation*}
\end{Definition}
For simplicity, in what follows, we denote the   spaces of all normal functions by $L_n^p(\mathbb{R};E)$. Moreover, in this section, we need the following
assumptions:
\begin{description}
\item[(A1)]if $h\in L_n^2((0,\infty);H^{\frac12}(\Gamma))\cap L_n^4((0,\infty);H^{\frac12}(\Gamma))$,
 $\partial_th\in L_n^2((0,\infty);H^{-\frac12}(\Gamma))$ and $\sup_{t\geq0}\|h\|_{H^{\frac12}(\Gamma)}$ suitable small, then we denote
 the symbol spaces by $\Sigma_0=\mathcal {H}(h)$;
\item[(A2)] if $h\in L_n^2((0,\infty);H^{\frac32}(\Gamma))$, $\partial_th\in L_n^2((0,\infty);H^{-\frac12}(\Gamma))$ and
$\sup_{t\geq0}\|h\|_{H^{\frac12}(\Gamma)}$ suitable small, we will consider the symbol space $\Sigma_1=\mathcal {H}(h)$.
\end{description}
Here, $\mathcal {H}(h)$ stands for  the hull of $h$.

In particular, in what follows, a natural phase space can be given by
\begin{equation*}
X=H\times L^2 ~(\text{or}~V\times H^1),\  (u,b)\in X.
\end{equation*}
Furthermore, in virtue of the global existence of weak (strong) solution, we can define the process associated with the solution to
(\ref{1.2})-(\ref{1.3}) acting in the phase spaces $X$ indexed by a symbol $\sigma\in \Sigma_0$ (or $\sigma\in \Sigma_1$).
\begin{Lemma}\label{le4.2}
Let $n=2$. Let all assumptions of Theorem \ref{th1.1} and (A1) be verified. Then the system  (\ref{1.2})-(\ref{1.3}) admits a uniform
(w.r.t. $\sigma\in\Sigma_0$) absorbing set $B_0\subset H\times L^{2}:$
\begin{equation*}
B_0=\{(u,b)\in H\times L^{2}:~ \|u\|_{H}^{2}+\|b\|_{L^2}^{2}\leq\rho_0\}.
\end{equation*}
where
\begin{equation*}
\rho_0=2\tilde{c}\|h\|^2_{L^{\infty}((0,\infty);L^2(\Omega))}+\frac{e^{c_p}c_0}{e^{c_p}-1}\left(\|h\|^2_{L_n^2(H^{\frac12}(\Gamma))}
+\|\partial_th\|^2_{L_n^2(H^{-\frac12}(\Gamma))}+\|h\|^4_{L_n^4(H^{\frac12}(\Gamma))}\right),
\end{equation*}
and the uniform (w.r.t. $\sigma\in\Sigma_0$) absorbing time of bounded set $B$ in $B_0$ is given by:
\begin{equation*}
t_0(B)=\frac1{c_p}\ln\frac{\text{diam}(B)}{\tilde{c}\|h\|^2_{L^{\infty}((0,\infty);L^2(\Omega))}}.
\end{equation*}
Moreover, for $t\geq t_0(B)$, there holds
\begin{equation}\label{4.2}
\int_{t}^{t+1}\|u(\tau)\|_{V}^{2}d\tau+\int_{t}^{t+1}\|b(\tau)\|_{H^1}^{2}d\tau\leq\rho_1,
\end{equation}
with
\begin{equation*}
\rho_1=(c_p+1+c_{\Omega})\rho_0,
\end{equation*}
where $c_p$, $c_0$, $\tilde{c}$ and $c_{\Omega}$ are positive constants defined in (\ref{4.3}), (\ref{4.4}), (\ref{4.5}) and (\ref{4.6}), respectively.
\end{Lemma}
\begin{proof}
Similar to (\ref{3.20}),  for the weak solution $(u,b)$, we have
\begin{align}\label{4.3}
&\frac{d}{dt}(\|u\|_{L^2}^{2}+\|\tilde{b}\|_{L^2}^{2})+\|\nabla u\|_{L^2}^{2}+\|\nabla\tilde{b}\|_{L^2}^{2}\nonumber\\
&\leq c_1\|h\|_{H^{\frac{1}{2}}(\Gamma)}^{4}(\|u\|_{L^2}^{2}+\|\tilde{b}\|_{L^2}^{2})+c_0(\|h\|_{H^{\frac{1}{2}}(\Gamma)}^{2}
+\|\partial_th\|_{H^{-\frac{1}{2}}(\Gamma)}^{2}+\|h\|_{H^{\frac{1}{2}}(\Gamma)}^{4}),
\end{align}
where $c_1$ and $c_0$ are two positive constants depend on $\Omega$. Since $\sup_{t\geq0}\|h\|_{H^{\frac12}(\Gamma)}$ suitable small, it holds that
\begin{equation}\label{4.4}
  c_1\sup_{t\geq0}\|h\|_{H^{\frac{1}{2}}(\Gamma)}^{4}\leq c_p,
\end{equation}
with $c_p=\frac12\min\{c_u,c_b\}$, and $c_u$, $c_b$ denote the Poincare's constant of $u$ and $\tilde{b}$, respectively, namely
\begin{equation*}
\|\nabla u\|_{L^2}^{2}\geq c_u\| u\|_{L^2}^{2},\quad \|\nabla\tilde{b}\|_{L^2}^{2}\geq c_b\|\tilde{b}\|_{L^2}^{2}.
\end{equation*}
Employing Gronwall's inequality, then from (\ref{4.3}), we deduce that
\begin{align}\label{4.5}
&\|u\|_{L^2}^{2}+\|b\|_{L^2}^{2}\leq\tilde{c}\|h\|^2_{L^{\infty}((0,\infty);L^2(\Omega))}\nonumber\\
& +e^{-c_pt}\left[(\|u_0\|_{L^2}^{2}+\|b_0\|_{L^2}^{2})+c_0\int_0^te^{c_p\tau}(\|h\|_{H^{\frac{1}{2}}(\Gamma)}^{2}
+\|\partial_th\|_{H^{-\frac{1}{2}}(\Gamma)}^{2}+\|h\|_{H^{\frac{1}{2}}(\Gamma)}^{4})d\tau\right],
\end{align}
where $\tilde{c}$ depends on $\Omega$. Thus, in order to obtain $B_0$, we only need  to prove that the integrals on the right hand side of
(\ref{4.5}) are bounded if $h\in \Sigma_0$. In fact, for any $t\geq0$, there exists $n\in \mathbb{N}$ such that $n-1\leq t\leq n$, and  we further obtain
\begin{align*}
 e^{-c_pn} \int_0^ne^{c_p\tau}\|h(\tau)\|_{H^{\frac{1}{2}}(\Gamma)}^{2}d\tau &
 \leq e^{-c_pn} \sum_{i=0}^{n-1}e^{c_p(i+1)}\int_{i}^{i+1}\|h(\tau)\|_{H^{\frac{1}{2}}(\Gamma)}^{2}d\tau \\
 & \leq e^{-c_pn}e^{c_p}\|h\|^2_{L_n^2(H^{\frac12}(\Gamma))} \sum_{i=0}^{n-1}e^{c_pi}\\
 & \leq \frac{e^{c_p}}{e^{c_p}-1} \|h\|^2_{L_n^2(H^{\frac12}(\Gamma))}.
\end{align*}
By the same way, we can also show that the rest two integrals are bounded from above. Thus, we obtain $B_0$ as claimed.
Now, we denote by $t_0(B)$ the absorbtion time of the bounded set $B$ in $B_0$, and $t_0$ can be derived from the following inequality
\begin{equation*}
e^{-c_pt}(\|u_0\|_{L^2}^{2}+\|b_0\|_{L^2}^{2})\leq \tilde{c}\|h\|^2_{L^{\infty}((0,\infty);L^2(\Omega))}.
\end{equation*}
In addition, note that
\begin{align}\label{4.6}
  \|b\|^2_{H^1} &\leq \|b\|^2_{L^2}+ \|\nabla \tilde{b}\|^2_{L^2}+ \|\nabla h_E\|^2_{L^2}\nonumber\\
  & \leq c_{\Omega} \|h\|^2_{H^{\frac12}(\Gamma)}+\|\nabla \tilde{b}\|^2_{L^2}.
\end{align}
Integrating  (\ref{4.3}) over $[t,t+1]$ with $t$ sufficiently large ($t\geq t_0(B)$), then we have (\ref{4.2}). Thus, we have completed the
proof of Lemma \ref{le4.2}.
\end{proof}
Similarly, based on the existence of global strong solution in Theorem \ref{th1.2}, we are able to prove the existence of absorbing sets bounded in more regular spaces $V\times H^1$.
\begin{Lemma}\label{le4.3}
Let all assumptions of Theorem \ref{th1.2} and (A2) be in force. Then the system (\ref{1.2})-(\ref{1.3}) admits a uniform
(w.r.t. $\sigma\in \Sigma_1$) absorbing set $B_2\in V\times H^1$:
\begin{equation*}
B_2=\{(u,b)\in V\times H^1:~  \|u\|_{H^1}^{2}+\|b\|_{H^1}^{2}\leq\rho_2\},
\end{equation*}
and a uniform (w.r.t. $\sigma\in \Sigma_1$) absorbing time for the bounded set $B$ in $B_2$ given by $t_2(B)=t_0(B)+1$.
Moreover, there holds
\begin{equation}\label{4.7}
\int_{t}^{t+1}\|u(\tau)\|_{H^2}^{2}d\tau+\int_{t}^{t+1}\|b(\tau)\|_{H^2}^{2}d\tau\leq\rho_3,
\end{equation}
where $\rho_2$ and $\rho_3$ depend on $\Omega$, $\|h\|_{L_n^2(\mathbb{R}_+;H^{\frac32}(\Gamma))}$ and
$\|\partial_th\|_{L_n^2(\mathbb{R}_+;H^{-\frac12}(\Gamma))}$.
\end{Lemma}
\begin{proof}
Taking into account (\ref{3.28}), applying the uniform Gronwall's inequality (cf. Chap.3 Sec.1.1.3 in \cite{bo30}), then for all
$\varepsilon,t\geq0$:
\begin{align*}
&\|\nabla u(t+\varepsilon)\|_{L^2}^2+ \|\nabla \hat{b}(t+\varepsilon)\|_{L^2}^2
 \leq \left(\frac{1}{\varepsilon}\int_t^{t+\varepsilon}(\|\nabla u(\tau)\|_{L^2}^2+\|\nabla \hat{b}(\tau)\|_{L^2}^2)d\tau\right.\\
&\left.+c\int_t^{t+\varepsilon}[(\|u(\tau)\|_{L^2}^2+\|b(\tau)\|_{L^2}^2)\|h(\tau)\|^4_{H^{\frac12}(\Gamma)}
+\|h(\tau)\|^2_{H^{\frac32}(\Gamma)}] d\tau \right)\times exp\left(\int_{t}^{t+\varepsilon}K(\tau)d\tau\right).
\end{align*}
In virtue of Lemma \ref{le2.3} and by choosing $\varepsilon=1$, then we obtain the existence of the absorbing set $B_2$.

Finally, by  Lemma \ref{le2.3}  and integrating (\ref{3.28}) from $t$ to $t+1$ with $t\geq t_2(B)$, then we have (\ref{4.7}). This, completes the proof
of Lemma \ref{le4.3}.
\end{proof}
\subsection{Existence of a uniform attractor}
In this section, we proceed to prove the existence of a uniform attractor for (\ref{1.2})-(\ref{1.3}).
\begin{proof}[Proof of Theorem \ref{th1.3}]
Recalling Theorem \ref{th4.1}, in order to prove Theorem \ref{th1.3}, we only need to prove $\omega$-limit compactness and weak continuity of a family of
process  $\{U_{h}(t,\tau)\}$. For simplicity, we divide the proof into several steps.

\textbf{Step 1.} $\omega$-limit compactness of $\{U_{h}(t,\tau)\}$.
Taking into account Lemma \ref{le4.1}, which provides a straightforward way to prove $\omega$-limit compactness of the process. First, by Lemma
\ref{le4.2} and Lemma \ref{le4.3}, the condition ($i_1$) is verified clearly. Next, we aim to check ($i_2$). Let $V^n$ be a subspace of $V$
for the velocity given by Proposition  \ref{po3.1}, $D^m$ be a space spanned by the first $m$ eigenfunctions of the Laplace's problem with homogeneous
Dirichlet boundary conditions in $\Omega$. Let $\{\lambda_n\}$ and $\{\mu_m\}$ be the eigenvalues of Stokes's problem and Laplace's problem in $\Omega$,
respectively. It is well known that $0<\lambda_1<\lambda_2<\cdots\nearrow\infty$ and $0<\mu_1<\mu_2<\cdots\nearrow\infty$ are monotone increasing
sequences. In what follows, we use $P_n$ and $Q_m$ as projections on $V^n$ and $D^m$, respectively. Moreover, consider the following lifted approximate
problems
\begin{equation}\label{4.12}
\begin{cases}
\partial_t u_m-\Delta u_m+u_m\cdot\nabla u_m-b_m\cdot\nabla b_m+\nabla p_m=0 &\text{in}~ Q_T,\\
\partial_t \tilde{b}_m-\Delta\tilde{b}_m+u_m\cdot\nabla (\tilde{b}_m+h_E)-(\tilde{b}_m+h_E)\cdot\nabla u_m+\partial_t h_E=0 &\text{in}~ Q_T,\\
\textrm{div}u_m=0&\text{in}~ Q_T,\\
u_m(0)=P_mu_{0},~\tilde{b}_m(x,0)=Q_m(b_0-h_E(0))  &\text{in}~\Omega,\\
u_m(x,t)=\tilde{b}_m(x,t)=0 &\text{on}~\Gamma_T,
\end{cases}
\end{equation}
or
\begin{equation}\label{4.13}
\begin{cases}
\partial_t u_m-\Delta u_m+u_m\cdot\nabla u_m-b_m\cdot\nabla b_m+\nabla p_m=0 &\text{in}~ Q_T,\\
\partial_t \hat{b}_m-\Delta\hat{b}_m+u_m\cdot\nabla (\hat{b}_m+h_p)-(\hat{b}_m+h_p)\cdot\nabla u_m+\partial_t h_p=0 &\text{in}~ Q_T,\\
\textrm{div}u_m=0&\text{in}~ Q_T,\\
u_m(0)=P_mu_{0},~\hat{b}_m(x,0)=Q_m(b_0-h_p(0))  &\text{in}~\Omega,\\
u_m(x,t)=\hat{b}_m(x,t)=0 &\text{on}~\Gamma_T.
\end{cases}
\end{equation}
Analogous to the proof of Theorem \ref{th1.1}-\ref{th1.2}, by (\ref{4.12})-(\ref{4.13}), we can obtain the existence of weak and strong solutions
$(u,b)$ to (\ref{1.2})--(\ref{1.3}). At this stage, we define $u_1:=P_nu$,
$b_1:=Q_m \tilde{b}$, $u_2:=u-u_1$ and $b_2:=\tilde{b}-b_1$ with $\tilde{b}=b-h_E$.

Multiplying (\ref{1.3})$_1$ and (\ref{1.3})$_2$ with $-S u_2=\Delta u_2-\nabla p_2$ and $-\Delta b_2$, respectively, we can see that
\begin{align}\label{4.8}
\frac{1}{2}&\frac{d}{dt}(\|\nabla u_2\|_{L^2}^{2}+\|\nabla b_2\|_{L^2}^{2})+\|S u_2\|_{L^2}^{2}+\|\Delta b_2\|_{L^2}^{2} \nonumber \\
=&\int_{\Omega} (u\cdot\nabla u)\cdot S u_2dx-\int_{\Omega} (b\cdot\nabla b)\cdot S u_2dx\nonumber \\
 & +\int_{\Omega} (u\cdot\nabla b)\cdot\Delta b_2dx-\int_{\Omega} (b\cdot\nabla u)\cdot\Delta b_2dx
 +\int_{\Omega}\partial_th_E\cdot\Delta b_2dx\nonumber \\
:=&R_1+R_2+R_3+R_4+R_5,
\end{align}
where in the left hand side of (\ref{4.8}), we have taken into account
\begin{equation*}
\int_{\Omega}Su_1\cdot Su_2 dx=0\quad \text{and}\quad \int_{\Omega}\nabla (p-p_1-p_2)\cdot Su_2 dx=0,
\end{equation*}
with $p_1,$ $p_2$ are the pressure terms corresponding to $u_1$, $u_2$ respectively, satisfying
\begin{align*}
Su_1&=-\Delta u_1+\nabla p_1=g_i(t)\sum_{i=1}^n\lambda_i\xi_i(x),  \\
Su_2&=-\Delta u_2+\nabla p_2=g_i(t)\sum_{i=n+1}^{\infty}\lambda_i\xi_i(x).
\end{align*}
From Lemma \ref{le4.2}, Lemma \ref{le4.3} and Lemma \ref{le2.4}, it follows that
\begin{align*}
R_1= &\int_{\Omega} (u_1\cdot\nabla u)\cdot S u_2dx+\int_{\Omega} (u_2\cdot\nabla u)\cdot S u_2dx  \\
\leq& \|u_1\|_{L^{\infty}}\|\nabla u\|_{L^2}\|S u_2\|_{L^2}+\|u_2\|_{L^{\infty}}\|\nabla u\|_{L^2}\|S u_2\|_{L^2}\\
\leq& c \|\nabla u_1\|_{L^2}\left(1+\ln\frac{\|\Delta u_1\|^2_{L^2}}{\|\nabla u_1\|^2_{L^2}}\right)\|\nabla u\|_{L^2}\|S u_2\|_{L^2}
 +c\| u_2\|_{L^2}^{\frac12}\|S u_2\|_{L^2}^{\frac12}\|\nabla u\|_{L^2}\|S u_2\|_{L^2}\\
\leq& c\rho_2(1+\ln [(c_0+1)\lambda_{n+1}])^{\frac12}\|S u_2\|_{L^2}+c\rho_0^{\frac12}\rho_2^{\frac12}\|S u_2\|_{L^2}^{\frac32}\\
\leq& c(\rho_2^2(1+\ln [(c_0+1)\lambda_{n+1}])+\rho_0^2\rho_2^2)+\frac18\|S u_2\|_{L^2}^2,
\end{align*}
where in the second inequality, we have used  the equivalent norms $\|v\|_{H^2}\approx \|\Delta v\|_{L^2}$ in $H^1_0(\Omega)\cap H^2(\Omega)$,
and the fact
\begin{equation}\label{4.08}
\|\Delta u_1\|_{L^2}^2\leq (c_0+1)\lambda_{n+1}\|\nabla u_1\|_{L^2}^2,
\end{equation}
where $c_0$ only depends on $\Omega$ and the spatial dimension. In fact, in view of (\ref{3.01}) and (\ref{3.014}), we can see that
\begin{equation*}
\|\Delta u_1\|_{L^2}^2=\sum_{i=1}^n|g_i(t)|^2\|\lambda_i\xi_i+\nabla p_i\|_{L^2}^2
= \sum_{i=1}^n|g_i(t)|^2(\lambda_i^2\|\xi_i\|_{L^2}^2+\|\nabla p_i\|_{L^2}^2),
\end{equation*}
and
\begin{align*}
\|\nabla u_1\|_{L^2}^2& =\sum_{i=1}^n|g_i(t)|^2\int_{\Omega}\nabla\xi_i\cdot\nabla\xi_idx\\
&=\sum_{i=1}^n|g_i(t)|^2\int_{\Omega}-\Delta\xi_i\cdot\xi_idx
= \sum_{i=1}^n|g_i(t)|^2\lambda_i\|\xi_i\|_{L^2}^2.
\end{align*}
Thus, combining these two conclusions and Lemma \ref{le2.1}, which easily yields (\ref{4.08}).

Furthermore, note that $\|\Delta b_1\|_{L^2}^2\leq \mu_{m+1}\|\nabla b_1\|_{L^2}^2$, by H\"{o}lder, Young and Sobolev's inequalities, we  obtain
\begin{align*}
R_2&\leq \|b\|_{L^4}\|\nabla b\|_{L^4}\|S u_2\|_{L^2}\\
&\leq c\|b\|_{L^2}^{\frac{1}{2}}\|b\|_{H^1}^{\frac{1}{2}}
\|\nabla b\|_{L^2}^{\frac{1}{2}}\|\nabla b\|_{H^1}^{\frac{1}{2}}\|S u_2\|_{L^2}\\
&\leq c\rho_{0}^{\frac{1}{4}}\rho_{2}^{\frac{1}{2}}(\|\Delta\tilde{b}\|_{L^2}^{\frac{1}{2}}+\|\nabla h_E\|_{H^1}^{\frac{1}{2}})\|S u_2\|_{L^2}\\
&\leq c\rho_{0}^{\frac{1}{4}}\rho_{2}^{\frac{1}{2}}(\|\Delta b_1\|_{L^2}^{\frac{1}{2}}+
\|\Delta b_2\|_{L^2}^{\frac{1}{2}}+\|h\|_{H^{\frac{3}{2}}(\Gamma)}^{\frac{1}{2}})\|S u_2\|_{L^2}\\
&\leq c\rho_{0}^{\frac{1}{4}}\rho_{2}^{\frac{1}{2}}(\mu_{m+1}^{\frac{1}{4}}\|\nabla b_1\|_{L^2}^{\frac{1}{2}}
+\|\Delta b_2\|_{L^2}^{\frac{1}{2}}+\|h\|_{H^{\frac{3}{2}}(\Gamma)}^{\frac{1}{2}})\|S u_2\|_{L^2}\\
&\leq c(\rho_0,\rho_2)\mu_{m+1}^{\frac{1}{2}}+c\rho_0\rho_2^2+c\|h\|_{H^{\frac{3}{2}}(\Gamma)}^2
+\frac18\|S u_2\|_{L^2}+\frac18\|\Delta b_2\|_{L^2}.
\end{align*}
Similarly, we further obtain
\begin{align*}
R_3&\leq \|u\|_{L^4}\|\nabla b\|_{L^4}\|\Delta b_2\|_{L^2}\\
&\leq c\|u\|_{L^2}^{\frac{1}{2}}\|\nabla u\|_{L^2}^{\frac{1}{2}}
\|\nabla b\|_{L^2}^{\frac{1}{2}}\|\nabla b\|_{H^1}^{\frac{1}{2}}\|\Delta b_2\|_{L^2}\\
&\leq c\rho_{0}^{\frac{1}{4}}\rho_{2}^{\frac{1}{2}}(\|\Delta\tilde{b}\|_{L^2}^{\frac{1}{2}}+\|\nabla h_E\|_{H^1}^{\frac{1}{2}})\|\Delta b_2\|_{L^2}\\
&\leq c\rho_{0}^{\frac{1}{4}}\rho_{2}^{\frac{1}{2}}(\|\Delta b_1\|_{L^2}^{\frac{1}{2}}
+\|\Delta b_2\|_{L^2}^{\frac{1}{2}}+\|h\|_{H^{\frac{3}{2}}(\Gamma)}^{\frac{1}{2}})\|\Delta b_2\|_{L^2}\\
&\leq c\rho_{0}^{\frac{1}{4}}\rho_{2}^{\frac{1}{2}}(\mu_{m+1}^{\frac{1}{4}}\|\nabla b_1\|_{L^2}^{\frac{1}{2}}
+\|\Delta b_2\|_{L^2}^{\frac{1}{2}}+\|h\|_{H^{\frac{3}{2}}(\Gamma)}^{\frac{1}{2}})\|\Delta b_2\|_{L^2}\\
&\leq c(\rho_0,\rho_2)\mu_{m+1}^{\frac{1}{2}}+c\rho_0\rho_2^2+c\|h\|_{H^{\frac{3}{2}}(\Gamma)}^2+\frac18\|\Delta b_2\|_{L^2},
\end{align*}
and
\begin{align*}
R_4&\leq \|b\|_{L^4}\|\nabla u\|_{L^4}\|\Delta b_2\|_{L^2}\\
&\leq c\|b\|_{L^2}^{\frac{1}{2}}\|b\|_{H^1}^{\frac{1}{2}}
\|\nabla u\|_{L^2}^{\frac{1}{2}}\|\Delta u\|_{L^2}^{\frac{1}{2}}\|\Delta b_2\|_{L^2}\\
&\leq c\rho_{0}^{\frac{1}{4}}\rho_{2}^{\frac{1}{2}}(\|\Delta u_1\|_{L^2}^{\frac{1}{2}}+\|S u_2\|_{L^2}^{\frac{1}{2}})\|\Delta b_2\|_{L^2}\\
&\leq c\rho_{0}^{\frac{1}{4}}\rho_{2}^{\frac{1}{2}}(\lambda_{n+1}^{\frac{1}{4}}\|\nabla u_1\|_{L^2}^{\frac{1}{2}}
+\|S u_2\|_{L^2}^{\frac{1}{2}})\|\Delta b_2\|_{L^2}\\
&\leq c\rho_0^{\frac12}\rho_1^{\frac32}\lambda_{n+1}^{\frac12}+c\rho_0\rho_2^2
+\frac18\|S u_2\|_{L^2}+\frac18\|\Delta b_2\|_{L^2}.
\end{align*}
Finally, for the term $R_5$, it is obvious that
\begin{equation*}
R_5\leq \|h_E\|_{L^2}\|\Delta b_2\|_{L^2}\leq c\|\partial_th\|^2_{H^{-\frac12}(\Gamma)}+\frac18\|\Delta b_2\|_{L^2}^2.
\end{equation*}
Putting these estimates into (\ref{4.8}) and taking into account Lemma \ref{le2.1}, we conclude that
\begin{align*}
&\frac{d}{dt}(\|\nabla u_2\|_{L^2}^{2}+\|\nabla b_2\|_{L^2}^{2})+\|\Delta u_2\|_{L^2}^{2}+\|\Delta b_2\|_{L^2}^{2} \\
&\leq c(\rho_0,\rho_2)(1+\ln[(c_0+1)\lambda_{n+1}]+\lambda_{n+1}^{\frac12}+\mu_{m+1}^{\frac12})+c\|h\|^2_{H^{\frac32}(\Gamma)}.
\end{align*}
Note that $\|\Delta u_2\|_{L^2}^2\geq \lambda_{n+1}\|\nabla u_2\|_{L^2}^2 $, $\|\Delta b_2\|_{L^2}^2\geq \mu_{m+1}\|\nabla b_2\|_{L^2}^2 $. Thus,
by Gronwall's inequality, the previous inequality implies that
\begin{align}\label{4.9}
\|\nabla u_2\|_{L^2}^{2}+\|\nabla b_2\|_{L^2}^{2}\leq &(\|\nabla u_2(t_0)\|_{L^2}^{2}+\|\nabla b_2(t_0)\|_{L^2}^{2})e^{-\gamma(t-t_0)} \nonumber \\
 & +\frac{c(\rho_0,\rho_2)}{\gamma}(1+\ln[(c_0+1)\lambda_{n+1}]+\lambda_{n+1}^{\frac12}+\mu_{m+1}^{\frac12})
 +\int_{t_0}^te^{-\gamma(t-s)}\|h\|^2_{H^{\frac32}(\Gamma)}ds,
\end{align}
where $\gamma=\min\{\lambda_{n+1},\mu_{m+1}\}$. Now we choose $n$ and $m$ sufficiently large such that $\lambda_{n+1}\approx \mu_{m+1}$,
then all terms on the right hand side of (\ref{4.9}) can be arbitrarily small, that is ($i_2$). Thus, we have proved the $\omega$-limit compactness
of the process.

\textbf{Step 2.} Weak continuity of the process $\{U_h(t,\tau)\}$. Now, we focus our attention on proving weak continuity of the process $\{U_h(t,\tau)\}$
with respect to initial data and boundary data $h\in \Sigma_1$.

Let $\{(u_{0n},b_{0n})\}\subset V\times H^1$, $(u_{0n},b_{0n})\longrightarrow(u_{0},b_{0})$ weakly in $V\times H^1$ and $\{h_n\}\subset\Sigma_1$,
$h_n\longrightarrow h$ weakly in $H^{\frac32}(\Gamma)$ be weakly convergent sequences of initial data and symbols. We propose to prove
$U_{h_n}(t,\tau)(u_{0n},b_{0n})\longrightarrow U_{h}(t,\tau)(u_{0},b_{0})$ weakly in $V\times H^1$. For this aim, we set
$(u_n(t),b_n(t))=U_{h_n}(t,\tau)(u_{0n},$ $b_{0n})$. Taking into account Lemma \ref{le4.3}, we infer that $\{(u_n(t),\tilde{b}_n(t))\}$ is bounded in
$L^{\infty}([\tau,\infty);V\times H^1)$ and in $L_{loc}^2([\tau,\infty);H^2\times H^2)$. Moreover, we can also obtain
$(\partial_tu_n,\partial_t\tilde{b}_n)$ is bounded in $L_{loc}^2([\tau,\infty);H\times L^2)$.

Next, we proceed to prove the  pre-compactness of the sequence $\{(u_n(t),\tilde{b}_n(t))\}$ in
$L^2_{loc}([\tau,\infty);V\times H^1)$. First, it is clearly that for all $v\in L^2$ and a.e. $t\in [\tau,T]$
\begin{align}\label{4.10}
\int_{\Omega}(u_n(t+\delta)-u_n(t))\cdot vdx &=\int_{t}^{t+\delta} \int_{\Omega}\partial_tu_n(s)\cdot vdxds\nonumber \\
 & \leq \delta^{\frac12}\|v\|_{L^2}\|\partial_tu_n\|_{L^2_{loc}([\tau,\infty);L^2)}\nonumber \\
  & \leq c\delta^{\frac12}\|v\|_{L^2},
\end{align}
where $\delta>0$ be suitable small constant.

Let $v=-\Delta(u_n(t+\delta)-u_n(t))$ in (\ref{4.10}), note that $\{u_n\}$ is bounded in $L^2([\tau,T-\delta];H^2)$, by integration by parts,
it holds that for all $T>\tau$
\begin{align}\label{4.11}
\int_{\tau}^{T-\delta}\|\nabla(u_n(t+a)-u_n(t))\|_{L^2}^2dt &\leq c \delta^{\frac12}\int_{\tau}^{T-\delta}
\|\Delta(u_n(t+\delta)-u_n(t))\|_{L^2}dt \nonumber \\
 &\leq c\delta^{\frac12}\left(\int_{\tau}^{T-\delta}\|\Delta(u_n(t+\delta)-u_n(t))\|^2_{L^2}dt\right)^{\frac12}\nonumber \\
 &\leq c(T)\delta^{\frac12}.
\end{align}
This implies that $\{u_n\}$ is pre-compact in $L^2([\tau,T-\delta];V)$ for all $T>\tau$.

Analogously, we can also obtain $\{\tilde{b}_n\}$ is pre-compact in $L^2([\tau,T-\delta];H_0^1)$, which together with the boundary condition easily yields
that $\{\tilde{b}_n\}$ is pre-compact in $L^2([\tau,T-\delta];H^1)$ for all $T>\tau$.

From the conclusion above, now, we can extract a subsequence of $\{(u_n,b_n)\}$, that converges to $(u,b)$ weakly in
$L_{loc}^2([\tau,\infty);H^2\times H^2)$, strongly in $L_{loc}^2([\tau,\infty);V\times H^1)$ and weak-star in
$L_{loc}^{\infty}([\tau,\infty);V\times H^1)$. Similar to Section \ref{se3}, we claim that $(u,b)$ indeed solves (\ref{1.2})-(\ref{1.3}).
Hence, for any regular pair $(v,w)$, we have for a.e. $t\geq \tau$
\begin{align*}
(\nabla u_n(t),v)&:=\int_{\Omega}\nabla u_n(t)\cdot vdx\longrightarrow \int_{\Omega}\nabla u(t)\cdot vdx,\\
(\nabla b_n(t),w)&:=\int_{\Omega}\nabla b_n(t)\cdot wdx\longrightarrow \int_{\Omega}\nabla b(t)\cdot wdx.
\end{align*}
Moreover, taking into account (\ref{4.10})-(\ref{4.11}), we can see that $(\nabla u_n(t),v)$ and $(\nabla b_n(t),w)$ are equibounded and equicontinuous
functions of $t$. This, together with the fact that the lifting problem (\ref{2.1}) is weakly continuous with respect to the boundary data, implies that
the weak continuity of the solution process.

Combining the conclusions above and Theorem \ref{th4.1}, one can deduce  the existence of a uniform attractor. Thus, we have completed the proof of
Theorem \ref{th1.3}.
\end{proof}
\bigbreak
\noindent{Acknowledgements.} The first and second author was supported by the National Natural Science Foundation of China (No.11726023, 11531010). The third author was supported by the  Postdoctoral Science Foundation of China (No. 2019TQ0006) and the Boya Postdoctoral Fellowship of Peking University.

\end{document}